\documentclass[twoside, paper=a4, fontsize=10pt]{scrartcl} 
\usepackage[colorlinks=true,linkcolor=blue,citecolor=blue,urlcolor=blue]{hyperref}
\usepackage[T1]{fontenc}

\usepackage[bitstream-charter]{mathdesign}

\usepackage[english]{babel}						
\usepackage[bbgreekl]{mathbbol} 
\usepackage{amsmath,amsfonts,amsthm}
\usepackage[pdftex]{graphicx}	
\usepackage{url}

\usepackage{bbm}

\usepackage{sectsty}		
\sectionfont{\centering \normalfont \scshape}

\subsectionfont{\normalfont \bfseries }

\usepackage{enumerate}

\usepackage{fancyhdr}
\usepackage[titletoc]{appendix}

\AtBeginDocument{}

\usepackage[numbers,comma,sort]{natbib}	

\numberwithin{equation}{section}

\newtheorem{theorem}{Theorem}[section]
\newtheorem{lemma}[theorem]{Lemma}
\newtheorem{corollary}[theorem]{Corollary}

\newtheorem{definition}[theorem]{Definition}
\newtheorem{proposition}[theorem]{Proposition}
\newtheorem{remark}[theorem]{Remark}

\begin{document}

\newcommand{\A}{\mathcal{A}}
\newcommand{\C}{\mathcal{C}}
\newcommand{\R}{\mathbb{R}}
\newcommand{\Q}{\mathbb{Q}}
\newcommand{\Z}{\mathbb{Z}}
\newcommand{\N}{\mathbb{N}}
\newcommand{\levy}{\mathscr{B}}
\newcommand{\sheet}{\mathbbm{B}}
\newcommand{\BB}{\boldsymbol{B}}
\newcommand{\alphar}{\texttt{$\boldsymbol{\alpha}$}}
\newcommand{\sigmar}{\texttt{\large $\boldsymbol{\sigma}$}}
\newcommand{\qA}{q_{\underline{\mathcal{A}}}}
\newcommand{\EE}{\mathbb{E}}
\newcommand{\PP}{\mathbb{P}}
\newcommand{\h}{\normalfont{\textit{\textbf{h}}}}

\title{\normalfont \Large \uppercase {A fractional Brownian field indexed by $L^2$ and a varying Hurst parameter}}

\author{\large Alexandre Richard\footnote{\hspace{0.1cm} INRIA-ECP Regularity Team  and Bar-Ilan University.\newline 
Ecole Centrale Paris, Laboratoire MAS, Grande Voie des Vignes, 92295 Châtenay-Malabry, France. \newline
E-mail: \href{mailto:alexandre.richard@ecp.fr}{alexandre.richard@ecp.fr}
}}

\date{}

\maketitle

\begin{abstract}
Using structures of Abstract Wiener Spaces, we define a fractional Brownian field indexed by a product space $(0,1/2] \times L^2(T,m)$, $(T,m)$ a separable measure space, where the first coordinate corresponds to the Hurst parameter of fractional Brownian motion. This field encompasses a large class of existing fractional Brownian processes, such as Lévy fractional Brownian motions and multiparameter fractional Brownian motions, and provides a setup for new ones. We prove that it has satisfactory incremental variance in both coordinates and derive certain continuity and Hölder regularity properties in relation with metric entropy. Also, a sharp estimate of the small ball probabilities is provided, generalizing a result on Lévy fractional Brownian motion. Then, we apply these general results to multiparameter and set-indexed processes, proving the existence of processes with prescribed local Hölder regularity on general indexing collections.
\end{abstract}



\section{Introduction and motivations}

The study of fractional processes began in the 1930's with the work of Kolmogorov in turbulent fluid dynamics. In the late 1960's, Mandelbrot and Van Ness popularized the notion of fractional Brownian motion (fBm). The family of processes $\{B^h_t, t\in\R_+\}$ is defined for each Hurst parameter $h\in(0,1)$, in such a way that $B^h$ is the only $h$-selfsimilar Gaussian process with stationary increments, and has Brownian motion ($h=1/2$) as a standard representant. These processes were extensively studied, and in this paper, we consider generalizations of fBm in two directions: \emph{a)} the family of fBm is considered for the different Hurst parameters as a single Gaussian process indexed by $(h,t)\in (0,1)\times \R_+$; \emph{b)} the ``time'' indexing is replaced by any separable $L^2$ space. We prove that there exists a Gaussian process indexed by $(0,1/2]\times L^2(T,m)$, with the additional constraint that the variance of its increments is as well behaved as it is on $(0,1)\times \R_+$.

\vspace{0.2cm}

The study of the first generalization originated in the works \cite{mBm1,mBm2} on what is now known as multifractional Brownian motion (mBm). The mBm can be introduced in a tractable way following the approaches of \cite{Dobric,AyacheTaqqu}, where a fractional Brownian field (fBf) is primarily defined. By fractional Brownian field, we will always mean a Gaussian process indexed by $t$ and $h$ simultaneously, and such that for fixed $h$, the process in $t$ is a fBm. A mBm is then built from a fBf and any given path in the $h$ direction, $\{h(t),t\in\R_+\}$. In \cite{AyacheTaqqu}, the authors use a wavelet series expansion of fBm to construct a fBf, while in \cite{Dobric}, the harmonizable integral representation of fBm is used. In both cases, harmonic analysis arguments allow to prove that for any compact subset of $\R_+$, there is a constant $C>0$ such that for any $t$ in this compact, and any $h,h'\in(0,1)$,
\begin{equation}\label{eq:Hincrement}
\EE\left(B^h_t-B^{h'}_t\right)^2 \leq C\ (h-h')^2 .
\end{equation}
This inequality is of some importance since it ensures the sample paths regularity of the field in $h$, while the behaviour with respect to the increments in $t$ is already known.

More generally, we will consider processes over $L^2(T,m)$, and an important subclass formed by processes restricted to indicator functions of subsets of $T$. In particular, multiparameter when $(T,m) = (\R_+^d, {\rm Leb.})$, and more largely set-indexed processes \cite{ehem,Ivanoff}, naturally appear and thus motivate generalization \emph{b)}, besides the inherent interest of studying processes over an abstract space. Therefore, our goal will be to construct a fractional Brownian field such that inequality (\ref{eq:Hincrement}) holds when $t$ is not in $\R_+$ anymore, but in some $L^2$ space. We shall write $L^2$--fBf for any such fractional Brownian field, or simply fBf if the context is clear, and $h$-fBm when looking at the $L^2$--fBf with a fixed $h$. A $h$-fBm will have the following covariance: for each $h\in(0,1/2]$,
\begin{align}\label{eq:dotprod}
k_h : (f,g)\in L^2\times L^2 \mapsto \frac{1}{2} \left(m(f^2)^{2h} + m(g^2)^{2h} -m(|f-g|^2)^{2h}\right) \ .
\end{align}
Note that according to Remark 2.10 of \cite{ehem}, $k_h$ is positive definite. $m(\cdot)$ on $L^2$ denotes the canonical linear functional associated to $m$, $\int_T (\cdot)^2 \ {\rm d}m$.

This form of covariance is particularly interesting for several reasons: it was thoroughly studied when restricted to indicator functions of some indexing collections, in particular in \cite{ehem,HerbinXiao}, where it is the covariance of the set-indexed fractional Brownian motion (SIfBm) and of the multiparameter fractional Brownian motion (mpfBm, a particular SIfBm indexed by rectangles of $\R^d$ with Lebesgue measure). Also, this covariance belongs to a larger class of functions on a metric space $(S,d)$, of the form: \[C(s,t) = 1/2 \left(d(t_0,t) + d(t_0,s) -d(t,s)\right)\] for $s,t\in S$, and an arbitrarily chosen origin $t_0\in S$. Whenever $d$ is such that $C(s,t)$ is positive definite (see, for instance, \cite{Takenaka} for a discussion), we call the resulting Gaussian process a Lévy Brownian motion, after Paul Lévy, who introduced it in the Euclidean setting \cite{PLevy}. Accordingly, the covariances will be said to be of the Lévy type. In the definition of $C(s,t)$, replacing $d(\cdot,\cdot)$ by $d(\cdot,\cdot)^{2H}$ for some $H\in(0,1]$ yields a covariance $C_H$, and thus a process referred to as Lévy fractional Brownian motion\footnote{To prove $C_H$ is positive definite, one can refer to \cite{Bernstein} where it is stated that the composition of a Bernstein function with a negative definite function is again negative definite.}. From this point of view, ${m(|\cdot - \cdot|^2)^{1/2}}$ is the $L^2$ metric, and $k_h$ is of the same form as $C_H$ (with $H=2h$). Since $C_H$ is positive definite for $H\in(0,1]$, it is coherent that $h\in (0,1/2]$ only. In the multiparameter setting (\cite{herbin}), the most studied fractional Brownian processes include the Lévy fBm, with covariance associated to the Euclidean distance: ${2 R_H(t,s) = \|t\|^{2H}+ \|s\|^{2H} - \|t-s\|^{2H}}$; and the fractional Brownian sheet, with covariance $2^d R_H(t,s) = \prod_{i=1}^{d}\{|t_i|^{2H_i} + |s_i|^{2H_i} - |t_i-s_i|^{2H_i}\}$. Interestingly in this setting, these covariances are not only of the Lévy type, but also of the form (\ref{eq:dotprod}). For the Lévy fBm, there exists a measure $\mu$ and a class of subsets $\{U_t, t\in \R_+^d\}$ of $\R^d$ such that:
\begin{align*}
\int_{\R^d} |\mathbf{1}_{U_s}- \mathbf{1}_{U_t}|^2 \ {\rm d}\mu = \mu(U_s\bigtriangleup U_t) = \|s-t\| \ .
\end{align*}
This is Centsov's construction (see a review in \cite[pp.400--402]{Samorodnitsky}). It is also possible to express the fractional Brownian sheet as a set-indexed Brownian motion (\cite{Bierme}), although the constructed product measure depends on $H$. We will explore in section \ref{subsec:spde} this construction, with an application to the regularity of solutions of a class of stochastic partial differential equations. It should now be clear that the form of covariance (\ref{eq:dotprod}) encompasses a wide class of processes.

\vspace{0.2cm}

In \cite{DecreusefondUstunel}, Decreusefond and \"Ust\"unel introduced a family of fractional operators on the Wiener space $W$ (i.e. the space of continuous functions on $[0,1]$, started at $0$), characterizing for each $h\in(0,1)$ a Cameron-Martin space $H_h$. Using these fractional operators, we express the fractional Brownian field as a white noise integral over the Wiener space:
\begin{align*}
\quad  \left\{\int_W \langle \mathcal{K}_h R_h(\cdot,t), w \rangle \ {\rm d}\mathbb{B}_w , \ (h,t)\in (0,1)\times[0,1] \right\} \ ,
\end{align*}
where $\mathbb{B}$ is the white noise associated to the standard Gaussian measure of $W$, $\mathcal{K}_h$ is derived from fractional operators appearing in \cite{DecreusefondUstunel}, $R_h$ is the covariance of the fBm, and $\langle \cdot, \cdot \rangle$ denotes the usual pairing between $W$ and its topological dual $W^*$. The advantage of this approach is to allow the transfer of techniques of calculus on the Wiener space to any other linearly isometric space with the same structure. Those spaces, called Abstract Wiener Spaces (AWS), were introduced by Gross in his seminal work \cite{Gross}. Using the separability and reproducing kernel property of the Cameron-Martin spaces built from the kernels ${k_h, h\in(0,1/2]}$, we prove the existence of a fractional Brownian field $\BB= {\{\BB_{h,f}, \ h\in(0,1/2], f\in L^2(T,m)\}}$ over some probability space $(\Omega,\mathcal{F},\PP)$. This is the topic of the second section, where the aforementioned notions are defined.

The third section is devoted to proving that the above $L^2$-fBf $\BB$ has good $h$-increments, as in (\ref{eq:Hincrement}). These results rely on Hilbert space analysis and analytic function theory, and are to be found first in Theorem \ref{prop:varInc1dfBf} for a generalised version of the fBf (in the sense of generalised processes \cite{Gelfand}) and then in Theorem \ref{th:regH} for the $L^2$-fBf. Some of the computations are reported in Appendix \ref{App:borneKh}. As an application of the first Theorem, we look at the $L^2(\Omega)$-continuity of the mild solutions of a class of stochastic partial differential equations (SPDE) with additive and anisotropic fractional noise, when the regularity of the noise changes. We remark that an interest in the continuity with respect to the Hurst parameter of some functionals of the fBm already appeared in the works of Jolis and Viles (see \cite{JolisViles} and previous works).

Then, in the fourth section, we use the increment properties of the variance of the $L^2$--fBf to derive a sufficient condition for almost sure continuity. We express in Theorem \ref{th:continuity} this condition under the form of a Dudley entropy integral which does not depend on the $h$ coordinate. This is an interesting application of the result of the previous part, since it means that many regularity properties of the fBf can be obtained from the sole observation of the $h$-fBm, for any fixed $h$. Another link with metric entropy is established in Theorem \ref{theo:smallBall} under the form of a sharp estimate of the small balls of the $h$-fBm. This is a natural extension of a result due to Monrad and Rootzén \cite{Monrad} for the fBm, and Talagrand \cite{Talagrand95} for the Lévy fBm. While doing so, a local nondeterminism property of this process is proved, similar to the one originally established by Pitt \cite{Pitt} in the 1970's.

We take a closer look at the H\"older regularity of the fBf in the fifth section, when the $L^2$ indexing collection is restricted to the indicator functions of the rectangles of $\R^d$ (multiparameter processes) or to some indexing collection (in the sense of \cite{Ivanoff}). This restriction permits to use local Hölder regularity exponents, in the flavour of what was done in \cite{ehar}. When a regular path $\h:L^2\rightarrow (0,1/2]$ is specified, this defines a multifractional Brownian field as $\BB^{\h}_f = \BB_{\h(f),f}$, whose Hölder regularity at each point is proved to equal $\h(f)$ almost surely.

\section{Fractional processes in an abstract Wiener space} 

Let us start with a few general remarks. $L^2(T,m)$ with its classical dot product $(\cdot,\cdot)_m$ will always be assumed to be separable. This is the case, for example, when $T$ is a locally compact metric space with a countable basis, and $m$ is a Borel measure (cf Chapter IV of \cite{Simon}).

\noindent We recall that it is impossible to construct a ``standard'' countably additive Gaussian measure on an infinite-dimensional Hilbert space (see, for instance, \cite{Kuo2}). By ``standard'', we mean that every one-dimensional cylindrical projection of this measure is a standard Gaussian measure over $\R$. In particular, describing the law of a Brownian motion indexed over $L^2(T,m)$ in terms of Gaussian measure is not straightforward. However, given a Hilbert space $H$ and a cylindrical measure $\mu$ on $H$, it is possible to embed this Hilbert space in a larger Banach space $E$ such that $\mu$ is countably additive on $E$, as this will be exposed in the next paragraph. The most natural process obtained from this construction is a Brownian process indexed by $H$. In order to produce fractional variations of Brownian motion, we will make use of special Hilbert spaces on which covariance functions can be decomposed:
\begin{definition}[Reproducing Kernel Hilbert Space]\label{def:RKHS}
Let $(T,m)$ be a separable and complete metric space and $R$ a continuous covariance function on $T\times T$. $R$ determines a unique Hilbert space $H(R)$ satisfying the following properties: \emph{i)} $H(R)$ is a space of functions on $T\rightarrow \R$; \emph{ii)} for all $t\in T$, $R(\cdot,t) \in H(R)$; \emph{iii)} for all $t\in T$, $\forall f\in H(R)$, $\left(f,R(\cdot,t)\right)_{H(R)} = f(t)$ .
\end{definition}
$H(R)$ can be constructed from $\textrm{Span}\{R(\cdot,t), t\in T\}$, completing this space with respect to the norm given by the scalar product of the previous Definition. The continuity of the kernel and the separability of $T$ suffice to prove that $H(R)$ is itself separable \cite{Borell76}.

\vspace{0.3cm}

We now present the construction of Gross \cite{Gross} of an Abstract Wiener space on $H$ equipped with its scalar product $(\cdot,\cdot)_H$. Let $\tilde{\mu}$ be the following cylindrical measure: for any cylindrical subset $S\subset H$, i.e. of the form $S = P^{-1}(B)$, where $B$ is a Borel subset of $H$ and $P$ an orthogonal projection of $H$ with finite rank equal to $n$,
\begin{equation*}
\tilde{\mu}(S) = \frac{1}{(2\pi)^{n/2}} \int_B e^{-\|x\|_H^2/2} \ {\rm d}x \ .
\end{equation*}
The measure $\tilde{\mu}$ is centred Gaussian, but is not countably additive on $H$ when it is infinite-dimensional. The following definition allows to extend $\tilde{\mu}$ to a proper measure on a larger space. A \emph{measurable norm} is a norm $\|\cdot\|_1$ on $H$ such that for any $\varepsilon>0$, there exists an orthogonal projection $P_{\varepsilon}$ with finite rank such that for any finite-rank projection $P$ which is orthogonal to $P_{\varepsilon}$, the following holds: 
\begin{equation*}
\tilde{\mu}\big(\{x\in H:\|Px\|_1>\varepsilon\}\big) < \varepsilon \ .
\end{equation*}
If such a norm exists, we may call $E$ the completion of $H$ with respect to this norm. Then, $(E,\|\cdot\|_1)$ is a Banach space in which $(H,\|\cdot\|_H)$ is dense and such that the canonical injection is compact and continuous. The same relationship holds between their topological duals $E^*$ and $H^*$ (assimilated to $H$ in the following). The main result in \cite{Gross} then reads: $\tilde{\mu}$ extends to a countably additive measure $\mu$ on all the cylinders of $E$. 

From now on, the image of $x^*\in E^*$ by the canonical injection will be denoted $g_{x^*}\in H$. A major consequence of Gross's theorem is that there is a measure whose Fourier transform is given by:
\begin{equation}\label{eq:Fourier}
\forall x^*\in E^*, \quad \int_E e^{i \langle x^*,x\rangle} \ {\rm d}\mu(x) = e^{-\frac{1}{2}\|g_{x^*}\|_H^2} \ ,
\end{equation}
or, written in terms of the second moment:
\begin{equation*}
\forall x^*,y^*\in E^*, \quad \int_E \langle x^*,x\rangle\ \langle y^*,x\rangle \ {\rm d}\mu(x) = \left(g_{x^*},g_{y^*}\right)_H .
\end{equation*}

The triple $(H,E,\mu)$ is an \emph{abstract Wiener space} and $H$ is referred as \emph{Cameron-Martin space} of the process $\mu$. We present now two results concerning abstract Wiener spaces (AWS), the first stating that new abstract Wiener spaces can be easily constructed from others that already exist.
\begin{theorem}[\cite{Stroock}]\label{th:equivAWS}
Let $H$ and $H'$ be two separable Hilbert spaces and $F$ a linear isometry  from $H$ to $H'$. Assume that an AWS $(H,E,\mu)$ is given. Then, there exists a Banach space $E'\supset H'$ and a linear isometry $\tilde{F}:E\rightarrow E'$ whose restriction to $H$ is $F$ and $(H',E',\tilde{F}_*\mu)$ is an AWS ($\tilde{F}_*\mu$ denotes the push-forward measure of $\mu$ by $\tilde{F}$).
\end{theorem}
\noindent In particular, starting from the AWS of continuous functions on $[0,1]$ with the sup-norm and the Wiener measure, it is possible to construct a large class of AWS. However, this does not mean that all AWS are the same, sicne for a single Hilbert space, there can be an uncountable family of AWS. However, starting from a Banach space and a measure, there is a unique Cameron-Martin space explicitely constructed from it.
\begin{lemma}[\cite{Stroock}]\label{lemma:Stroock}
\begin{enumerate}
\item For any $x^* \in E^*$, there is a unique $g_{x^*} \in H$ such that $(g,g_{x^*})_H = \langle x^*,g\rangle$ for all $g\in H$ and the mapping $i: x^*\in E^* \mapsto g_{x^*}\in H$ is linear, continuous, injective and its image is dense in $H$.
\item\label{item2} If $x\in E\setminus H$, then $\sup_{\{x^*\in E^*: \|h_{x^*}\|_H\leq 1\}} \langle x^*,x\rangle = \infty$ and for any $g\in H$, the norm is given by
 ${\|g\|_H = \sup\{\langle x^*,g\rangle, \ x^*\in E^* \textrm{ and } \|g_{x^*}\|_{H}\leq 1\}}$.
\item There exists a sequence $(x_n^*)\in (E^*)^{\N}$ such that $(g_{x^*_n})_{n\in \N}$ is an orthonormal basis of $H$.
\end{enumerate}
\end{lemma}
\vspace{0.3cm}

Our approach in this section will be to identify abstract Wiener spaces related to the covariance functions of the fractional Brownian motion: \[\forall s,t\in \R_+, \quad R_h(s,t) = \frac{1}{2} \left(|s|^{2h}+ |t|^{2h} - |s-t|^{2h}\right) \ ,\] $h\in(0,1]$, and operators providing links between these different AWS. Then, each fractional Brownian motion is expressed as an integral in the standard Wiener space, eventually providing a Gaussian field in $t$ and $h$. The second step is to extend this object to another family of abstract Wiener spaces based on the reproducing kernel Hilbert space (RKHS) of $k_h$, relying heavily on Theorem \ref{th:equivAWS}. Finally, we prove the resulting field has a ``good'' covariance structure, in the sense of (\ref{eq:Hincrement}).

\subsection{\texorpdfstring{$\boldsymbol{R_h}$}{R\_h} in the standard Wiener space}

The standard Wiener space on $[0,1]$ is the triple consisting of the Banach space of continuous functions started at $0$, denoted $W$; the Cameron-Martin space $H^1$ of absolutely continuous functions started at $0$; and the Gaussian measure $\mathcal{W}$ on $W$, characterized by equation (\ref{eq:Fourier}) with appropriate change ($E=W$, $H=H^1$ and $\mu = \mathcal{W}$).

In fact, $H^1$ is also the space of real-valued functions $g$ on $[0,1]$ of the form ${g(t) = \int_0^t \dot g(s)\ {\rm d}s}$, where $\dot g \in L^2([0,1])$, and the Hilbert norm is $\|g\|_{H^1} = \|\dot g\|_{L^2}$.
Using the Riesz representation theorem on $C([0,1])$, $g_{x^*}$, as well as any $x^* \in C([0,1])^*$, can be expressed by: $\forall w\in C([0,1])$, ${x^*(w) = \int_0^1 w(t) \Lambda_{x^*}({\rm d}t)}$, with $\Lambda_{x^*}$ a finite signed Radon measure on $[0,1]$. Besides, this equality yields for the total variation of $\Lambda_{x^*}$: ${|\Lambda_{x^*}|([0,1]) = \|x^*\|}$. Thus, we shall assimilate $x^*$ with $\Lambda_{x^*}$, writing:
\begin{equation}\label{eq:dualite}
\forall x^*\in W^*, \forall w\in W, \quad \langle x^*,w \rangle = \int_{[0,1]} w(t) \ x^*({\rm d}t) .
\end{equation}

We now characterize the existence of a family of fractional Wiener spaces, as described in \cite{DecreusefondUstunel}: for any $h\in (0,1)$, there is a one-to-one operator $K_h$ acting on $L^2([0,1])$ satisfying the following properties:
\begin{enumerate}[1)]
\item \label{DU1} The space $H_h = K_h\left(L^2([0,1])\right)$ is a subspace of $W$. If we denote $(\cdot,\cdot)_{H_h}$ the scalar product ($\|\cdot\|_{H_h}$ the norm) that makes $H_h$ isometric to $L^2$, and $\mathcal{W}_h$ the Gaussian measure whose Fourier transform is characterized by $(\cdot,\cdot)_{H_h}$, then $(H_{h},W,\mathcal{W}_{h})$ is an AWS. For $h=1/2$: $(H_{1/2},W,\mathcal{W}_{1/2}) = (H^1,W,\mathcal{W})$.
\item\label{DU2} Let $K_h^*$ be the adjoint operator of $K_h$. When $K_h^*$ is restricted to $W^*$, the operator $K_h \circ K_h^*$ is the canonical injection from $W^*$ to $H_h$ and it is a kernel operator whose kernel is precisely $R_h$. As a consequence, we use the same notation for both the operator and the kernel.
\item\label{DU4} $K_h$ is a Hilbert Schmidt operator and has a kernel on $[0,1]^2$ which is denoted $K_h$ too, so that for any $g\in L^2([0,1])$, $K_h g(t) = \int_0^1 K_h(t,s) g(s) \ {\rm d}s$. We notice that while the kernel $R_h$ is symmetric, $K_h$ is not, as this will become clear.
\end{enumerate}
This is summarized, for all $h$, by:
\begin{equation*}
W^* \overset{K_h^*}{\longrightarrow} L^2([0,1]) \overset{K_h}{\longrightarrow} H_h \overset{R_h^*}{\longrightarrow} W .
\end{equation*}

More details on $K_h$ are given along this section and in Appendix \ref{App:borneKh}, especially its integral formula (\cite{DecreusefondUstunel,Nualart}), while the link with fractional integrals is clearly established in \cite{DecreusefondUstunel}. For $t\in [0,1]$, we denote by $\delta_t$ the Dirac measure at point $t$, considered here as an element of $W^*$.

\begin{theorem}\label{th:structOp}
Let $h\in (0,1)$ and $H(R_h)$ be the RKHS of $R_h$. Then, $H(R_h)\subseteq H_h$ and for any ${f,g\in H(R_h)}$, ${(f,g)_{H(R_h)} = (f,g)_{H_h}}$. Besides, there is a linear isometry $\tilde{J}_h^*$ from $W^*$ to itself such that for any $\eta, \nu \in W^*$, 
\begin{align}\label{eq:1}
\int_W \langle \tilde{J}_h^* \eta, w \rangle \ \langle \tilde{J}_h^* \nu, w \rangle \ {\rm d}\mathcal{W}(w) = (R_h \eta, R_h \nu)_{H_h} \ .
\end{align}
Furthermore, for all $s,t\in [0,1]$, 
\begin{align}\label{eq:L2}
R_h(s,t) &= (K_h^*\delta_s, K_h^*\delta_t)_{L^2} \nonumber \\
&= \left(R_h\delta_s, R_h \delta_t \right)_{H_h} \nonumber \\
&= \int_W \langle \mathcal{K}_h R_h(\cdot,s),w \rangle \ \langle \mathcal{K}_h R_h(\cdot,t),w \rangle \ {\rm d}\mathcal{W}(w) ,
\end{align}
with $\mathcal{K}_h: H_h \rightarrow W^*$ defined by the relationship $\mathcal{K}_h = \tilde{J}_h^* \circ R_h^{-1}$.

\end{theorem}

Before proving this result, consider the following immediate application. For a white noise $\mathbb{B}$ on $W$ with control measure $\mathcal{W}$ on the probability space $(\Omega, \mathcal{F}, \PP)$, the following formula defines a Gaussian random field on $(0,1)\times [0,1]$:
\begin{equation}\label{eq:GfBm1D}
\forall (h,t) \in (0,1)\times [0,1], \quad \BB_{h,t} = \int_W \langle \mathcal{K}_h R_h(\cdot,t), w \rangle \ {\rm d}\mathbb{B}_w \ .
\end{equation}
Indeed, according to equation (\ref{eq:L2}), the mapping $w\mapsto \langle \mathcal{K}_h R_h(\cdot,t),w \rangle$ belongs to $L^2(W,\mathcal{W})$.
In addition, the previous theorem shows that for fixed $h$, this process is a fractional Brownian motion. 

\begin{remark}
A similar two-parameter Gaussian field appeared for the first time in \cite{AyacheTaqqu,Dobric}, although it was not expressed as an integral over the Wiener space. We present in Corollary \ref{Cor:RepLaw} a different form for (\ref{eq:GfBm1D}).
\end{remark}

\begin{proof}
By the definition of point \ref{DU2}), $R_h$ is the operator $K_h\circ K_h^*$ mapping $W^*$ to $H_h$, with the kernel $R_h$ of the fractional Brownian motion. The assertion $R_h \delta_t = R_h(t,\cdot)$, for any $t\in [0,1]$, follows. As a consequence, $R_h(\cdot,t) \in H_h$ for all $t$, which in turns suffices to prove that $H(R_h)\subseteq H_h$, since $H(R_h)$ is the completion of $\textrm{Span}\{R_h(\cdot,t), t\in [0,1]\}$. For the same reason, proving that $(\cdot,\cdot)_{H(R_h)} = (\cdot,\cdot)_{H_h}$ on $H(R_h)$ amounts to show that for all $s,t\in [0,1]$, $(R_h(s,\cdot),R_h(t,\cdot))_{H(R_h)} = (R_h(s,\cdot),R_h(t,\cdot))_{H_h}$, where the first term in the equality is, by definition, equal to $R_h(s,t)$.
Then, we infer from point \ref{DU2}) (of the definition of the adjoint operator) that $K_h^*\eta$ can be evaluated in the following manner:
\begin{equation*}
\forall f\in L^2, \quad (f,K_h^*\eta)_{L^2} = \langle \eta, K_h f \rangle = \int_0^1 K_h f(t) \ \eta({\rm d}t) \ .
\end{equation*} 
Applied to $\delta_t$ and taking into account point \ref{DU4}), this yields:
\begin{align*}
(f,K_h^*\delta_t)_{L^2} = K_h f(t) &= \int_0^1 K_h(t,s) f(s) \ {\rm d}s
\end{align*}
\noindent Thus $K_h^*\delta_t = K_h(t,\cdot)$ in $L^2$ and it follows that $K_h\circ K_h^* \delta_t = \int_0^1 K_h(\cdot,r) \ K_h(t,r) \ {\rm d}r$
which is also equal to $R_h(t,\cdot)$, as mentioned at the beginning of the proof. According to point \ref{DU1}), $(\cdot,\cdot)_{H_h}$ satisfies, for $f,g\in H_h$, ${(f,g)_{H_h} = (K_h^{-1}f, K_h^{-1} g)_{L^2}}$. As a consequence, 
\begin{align*}
(R_h(\cdot,t),R_h(\cdot,s))_{H_h} &= (K_h^*\delta_t, K_h^*\delta_s)_{L^2} \\
& = \int_0^1 K_h(s,r) \ K_h(t,r) \ {\rm d}r\\
&= R_h(s,t)\ ,
\end{align*}
and the first result follows.

To prove the second point, let us define $J_h= K_h \circ K_{1/2}^{-1}$. From its definition, $J_h$ is an isometric isomorphism from $H_{1/2}$ towards $H_h$, thus admitting a unique (linear) isometric extension to $W$. Let $\tilde{J}_h$ be this extension and notice that the image space of $\tilde{J}_h$ is $\overline{H_h}^{\|\cdot\|_1}$, where $\|\cdot\|_1$ is the norm defined by $\|g\|_1 = \|J_h^{-1}(g)\|_W, g\in H_h$. It is clear that $\overline{H_h}^{\|\cdot\|_1} = \overline{H_{1/2}}^{\|\cdot\|_W} = W$, and as a consequence, $(H_h,W,\tilde{J}_{h \ *}\mathcal{W})$ is the image of the standard Wiener space by the isometry $J_h$. In particular, this identifies $\tilde{J}_{h \ *}\mathcal{W} = \mathcal{W}_h$ (these measures have the same Fourier transform). Let $\tilde{J}_h^*:W^*\rightarrow W^*$ be the adjoint operator of $\tilde{J}_h$. Then:
\begin{align*}
\int_W \langle \tilde{J}_h^* \eta, w \rangle \ \langle \tilde{J}_h^* \nu, w \rangle \ {\rm d}\mathcal{W}(w) &= \int_W \langle \eta, \tilde{J}_h w \rangle \ \langle \nu, \tilde{J}_h w \rangle \ {\rm d}\mathcal{W}(w) \\
&= \int_W \langle \eta, w \rangle \ \langle \nu,  w \rangle \ {\rm d}\left(\tilde{J}_{h\ *}\mathcal{W}\right)(w) \\
&= \int_W \langle \eta, w \rangle \ \langle \nu,  w \rangle \ {\rm d}\mathcal{W}_h(w) \\
&= \left(R_h \eta, R_h \nu\right)_{H_h} .
\end{align*}

Finally, (\ref{eq:L2}) directly follows from (\ref{eq:1}), using the fact that $R_h(\cdot,t) = R_h\delta_t$.
\end{proof}

To end this section, we show that the process defined by (\ref{eq:GfBm1D}) is equal to the process that appeared in \cite{DecreusefondUstunel}. This is no surprise, since the same operators are involved. However, we include the proof for completeness.

\begin{corollary}\label{Cor:RepLaw}
Let $\BB$ a process as defined in equation (\ref{eq:GfBm1D}). Then $\BB$ has the same law as the process $B_1 = \left\{ \int_0^1 K_h(t,s) \ {\rm d}W_s, \ (h,t)\in (0,1)\times[0,1] \right\}$.
\end{corollary}

\begin{proof}
Since both processes are centred Gaussian it suffices to study their covariance. 
For any $(h,t), (h',s) \in (0,1)\times[0,1]$, 
\begin{equation*}
C_1\left((h,t), (h',s)\right) = \int_0^1 K_h(t,u) \ K_{h'}(s,u) \ {\rm d}u = \left(K_h^* \delta_t, K_{h'}^* \delta_s\right)_{L^2} .
\end{equation*}
The covariance of $\BB$ is also derived, leading to:
\begin{equation*}
C_2\left((h,t), (h',s)\right) = \int_W \langle \tilde{J}_h^* \delta_t, w \rangle \langle \tilde{J}_{h'}^* \delta_s, w \rangle \ \mathcal{W}({\rm d}w) \ ,
\end{equation*}
and then, considering $\tilde{J}_h^*\delta_t \in W^*$ as a measure:
\begin{align*}
C_2\left((h,t), (h',s)\right) &= \int_W \left(\int_0^1 \int_0^1 w(u) w(v) \ (\tilde{J}_h^*\delta_t)({\rm d}u)\ (\tilde{J}_{h'}^*\delta_s)({\rm d}v) \right) \mathcal{W}({\rm d}w) \\
&= \int_0^1 \int_0^1 \left(\int_W w(u) w(v) \ \mathcal{W}({\rm d}w) \right) (\tilde{J}_h^*\delta_t)({\rm d}u)\ (\tilde{J}_{h'}^*\delta_s)({\rm d}v) \\
&= \int_0^1 \int_0^1 R_{1/2}(u,v)  \ (\tilde{J}_h^*\delta_t)({\rm d}u)\ (\tilde{J}_{h'}^*\delta_s)({\rm d}v) \\
&= \left(K_{1/2}^*\tilde{J}_h^*\delta_t, K_{1/2}^*\tilde{J}_{h'}^*\delta_s\right)_{L^2} .
\end{align*}
Hence, we are to prove that for any $h\in(0,1)$, $K_{1/2}^* \tilde{J}_h^*$ and $K_h^*$ are equal, as operators from $W^*$ to $L^2[0,1]$. The main ingredient is that when restricted to $H_{1/2}$, $\tilde{J}_h$ is equal to $K_h\circ K_{1/2}^{-1}$. Then, for any $x^*\in W^*, f\in L^2[0,1]$,
\begin{align*}
\left(K_{1/2}^* \tilde{J}_h^* x^*, f\right)_{L^2} &= \langle \tilde{J}_h^* x^*, K_{1/2} f\rangle \\
&= \langle x^* , \tilde{J}_h K_{1/2} f \rangle \\
&= \langle x^* , J_h K_{1/2} f \rangle  \\
&= \langle x^*, K_h f \rangle ,
\end{align*}
where the third equality holds because $K_{1/2} f\in H_{1/2}$.
\end{proof}

\begin{remark}
We make a final remark in this section, with the definition of an extension of $\{\BB_{h,t}, (h,t)\in (0,1)\times [0,1]\}$ (or $B_1$) to a generalised process (in the sense of \cite{Gelfand}): $\{\BB_{h,\xi}, (h,\xi) \in (0,1)\times W^*\}$, where following (\ref{eq:1}), 
\begin{equation}\label{eq:1dgfBf}
\BB_{h,\xi} = \int_W \langle \tilde{J}_h^* \xi, w \rangle \ {\rm d}\mathbb{B}_w \ .
\end{equation}
\end{remark}

\subsection{Extended decomposition in an abstract Wiener space}\label{subsec:extended}

Equipped with a tractable way of expressing a fractional Gaussian field given the reproducing kernel $R_h$ on $[0,1]$, we now present how to extend this decomposition to any separable $L^2(T,m)$ space with the family of kernels $k_h, h\in (0,1/2]$ introduced in (\ref{eq:dotprod}). The RKHS of $k_h$ is written $H(k_h)$.

For any $h\in (0,1/2]$, $H_h \supset H(R_h)$ and $H(k_h)$ are separable Hilbert spaces, so let us choose a linear isometry between the two latter spaces and extend it to $H_h$. We call $u_h$ such a linear isometry and write $\mathcal{H}_h = u_h(H_h)$. As a consequence of its definition, the restriction of $u_h$ to $H(R_h)$ is a linear isometry between $H(R_h)$ and $H(k_h)$. Recall that the parameter $h$ is restricted to $(0,1/2]$ because $k_h$ is not positive definite for $h\in(1/2,1]$ (see \cite{ehem} for a counterexample). In the next section, we will be more specific about the choice of this isometry. $u_h$ is isometrically extended to all of $W$ as in Theorem \ref{th:equivAWS}, and the extension is denoted by $\tilde{u}_h$. By this very Theorem, it is possible to define an abstract Wiener space which is the image of $(H_h,W,\mathcal{W}_h)$ by $u_h$, and we denote it $(\mathcal{H}_h, E_h, \mu_h)$, where $E_h=\tilde{u}_h(W)$ and $\mu_h = (\tilde{u}_h)_*\mathcal{W}_h$, the pushforward measure of $\mathcal{W}_h$ by $\tilde{u}_h$. The adjoint operator $\tilde{u}_h^T$ of $\tilde{u}_h$ is the mapping from $E_h^*$ to $W^*$ such that for any $w\in W$, any $w^*\in E_h^*$, $\langle w^*,\tilde{u}_h(w) \rangle = \langle \tilde{u}_h^T(w^*), w\rangle$. Since the space associated to $h=1/2$ plays a special part, we just drop the $h$ in the notations. Especially, $(H,W,\mathcal{W})$ denotes the standard Wiener space.

Now define $\tilde{\mathcal{K}}_h = (\tilde{u}^T)^{-1} \circ \mathcal{K}_h \circ u_h^{-1}$, the linear operator from $\mathcal{H}_h$ to $E^*$. For $\phi,\psi \in \mathcal{H}_h$,
\begin{align*}
(\phi, \psi)_{\mathcal{H}_h} &= (u_h^{-1}\phi, u_h^{-1}\psi)_{H_h} = \int_W \langle\mathcal{K}_h u_h^{-1}\phi,w \rangle \ \langle\mathcal{K}_h u_h^{-1}\psi,w \rangle \ {\rm d}\mathcal{W}(w)\\
& = \int_E \langle\mathcal{K}_h u_h^{-1}\phi, \tilde{u}^{-1}(x) \rangle \ \langle\mathcal{K}_h u_h^{-1}\psi,\tilde{u}^{-1}(x) \rangle \ {\rm d}\left(\tilde{u}_*\mathcal{W}\right)(x)\\
& = \int_E \langle\tilde{\mathcal{K}}_h \phi,x \rangle \ \langle\tilde{\mathcal{K}}_h \psi,x \rangle \ {\rm d}\mu(x) \ .
\end{align*}

\noindent When applied to $\phi = k_h(\cdot,f)$ and $\psi = k_h(\cdot,g)$, for $f,g\in L^2(T,m)$, the previous relation reads:
\begin{equation*}
k_h(f,g) = \int_E \langle\tilde{\mathcal{K}}_h k_h(\cdot,f),x \rangle \ \langle\tilde{\mathcal{K}}_h k_h(\cdot,g),x \rangle \ {\rm d}\mu(x) \ .
\end{equation*}

\begin{definition}[fractional Brownian field]
Let $(\Omega, \mathcal{F}, \PP)$ a probability space and $\mathbb{W}$ a white noise on $E$ associated to the measure $\mu$. The following formula defines a Gaussian random field over $(0,1/2]\times L^2(T,m)$:
\begin{equation*}%\label{eq:defSIfBm}
\forall (h,f) \in (0,1/2]\times L^2(T,m), \quad \BB_{h,f} = \int_E \langle \tilde{\mathcal{K}}_h k_h(\cdot,f), x \rangle \ \mathrm{d}\mathbb{W}(x) \ .
\end{equation*}
\end{definition}
The previous calculus proves that for fixed $h$, this process has covariance (\ref{eq:dotprod}), so $\BB$ is a $L^2$--fBf as defined in the introduction. Noticeably, if $\mathbb{W}^h$ is a white noise of $(\mathcal{H}_h,E_h,\mu_h)$, the process:
\begin{equation*}
\left\{\int_{E_h} \langle \ \mathbf{k}_h(\cdot,f), x\rangle \ {\rm d}\mathbb{W}^h_x, \ f\in L^2(T,m)\right\}
\end{equation*}
and $\left\{\BB_{h,f}, \ f\in L^2(T,m)\right\}$ have the same law, where $\mathbf{k}_h(\cdot,f) \in E_h^*$ is the pre-image of $k_h(\cdot,f)$ by the canonical injection of $E_h^*\rightarrow \mathcal{H}_h$. The same calculus as in the proof of Corollary \ref{Cor:RepLaw}, shows that the covariance of the new process is given by:
\begin{equation}\label{eq:VarL2fBf}
\EE\left( \BB_{h,f} \ \BB_{h',g} \right) = \left(K_h^{-1}u_h^{-1}k_h(f,.), K_{h'}^{-1}u_{h'}^{-1}k_{h'}(g,.)\right)_{L^2[0,1]} \ .
\end{equation}
Hence, the choice of the family of isometries $\{u_h\}_{h\in (0,1/2]}$ matters.

Let us set up a family of isometries $u = \{u_h\}_{h\in (0,1/2]}$ allowing the fBf to have good increments, as discussed already. Let $\mathbb{D} = \{t_n, n\in \N\}$ be the set of dyadics in $[0,1]$. Let $h\in(0,1/2]$, then from the definition of $H(R_h)$, $\{R_h(\cdot,t), t\in \mathbb{D}\}$ is a linear basis of this space (although the linear independence is proved in the following lemma). We shall establish the existence of a family of functions $\{f_n\in L^2(T,m), n\in \N\}$ such that $\{k_h(\cdot,f_n), n\in \N\}$ is a linear basis for $\mathcal{H}_h$.

\begin{lemma}\label{lem:LinInd}
Let $h\in (0,1/2]$. Let $n\in \N^*$ and $(f_0,\dots,f_n)\in L^2(T,m)$ be linearly independent. Then, $(k_h(\cdot,f_0),\dots, k_h(\cdot,f_n))$ is linearly independent in $\mathcal{H}_h$.
\end{lemma}

The proof is reported in Appendix \ref{App:lemLinInd}. This lemma suggests that we will choose a dense (countable) linear basis of $L^2(T,m)$. This can be done as follows. $L^2(T,m)$ is a separable metric space, hence it admits a countable topological basis\footnote{any separable metric space is second-countable.}, so denote by $(O_n)_{n\in\N}$ this basis of open sets for the topology of $L^2(T,m)$. Let us prove inductively the existence of a dense linearly independent family $(f_n)_{n\in \N}$. Let $f_0\in O_0$ and assume that $(f_0,\dots,f_n)$ already exist. ${F_n = \textrm{Span}\{f_0,\dots,f_n\}}$ is a finite-dimensional subspace of an infinite-dimensional space. Therefore $F_n$ is of empty interior and one can pick $f_{n+1}\in O_{n+1}\setminus F_n$. This mechanism provides a linearly independent family, which is dense since its intersection with any $O_n$ is non-empty.

From now on, $\left\{f_n, n\in \N \right\}$ denotes such a family, and for any $h\in (0,1/2]$, Lemma \ref{lem:LinInd} implies that $\left\{ k_h(\cdot,f_n), n\in \N \right\}$ is linearly independent in $H(k_h)$. That $\{f_n\}_{n\in \N}$ is dense in $L^2$ yields that this is actually a spanning family of $H(k_h)$. Let us prove this last statement:
let $K\in H(k_h)$, so that there is a sequence of real numbers $(\alpha_n)_{n\in \N}$ and a sequence $(g_n)_{n\in \N}$ of elements of $L^2(m)$, such that $\sum_{n\in\N} \alpha_n k_h(\cdot,g_n)$ is equal to $K$ in $H(k_h)$. Let $\varepsilon>0$ and for any $n\in \N$, let ${\varepsilon_n = \varepsilon\ |\alpha_n|^{-1}\ (n+1)^{-2}}$. For any $n$, there is $f_{i_n}$ in the dense basis of $L^2(m)$ that satisfies ${\|g_n - f_{i_n}\|_{L^2}^{2h} < \varepsilon_n}$. Let $N\in \N$ such that for all $p\geq N$, $\|K - \sum_{n=0}^p \alpha_n k_h(\cdot,g_n)\|_{H(k_h)} < \varepsilon$. Then:
\begin{align*}
\big\|K - \sum_{n=0}^p \alpha_n k_h(\cdot,f_{i_n}) \big\|_{\mathcal{H}_h} &\leq \big\|K - \sum_{n=0}^p \alpha_n k_h(\cdot,g_n) \big\|_{\mathcal{H}_h} + \big\|\sum_{n=0}^p \alpha_n k_h(\cdot,g_n) - \sum_{n=0}^p \alpha_n k_h(\cdot,f_{i_n}) \big\|_{\mathcal{H}_h}\\
&\leq \varepsilon + \sum_{n=0}^p |\alpha_n| \ \|k_h(\cdot,g_n) - k_h(\cdot,f_{i_n}) \|_{\mathcal{H}_h} \\
&\leq \varepsilon + \sum_{n=0}^p |\alpha_n| \ \|g_n-f_{i_n}\|_{L^2}^{2h} \\
&\leq 3\varepsilon \ .
\end{align*}

The same reasoning shows that $\left\{ R_h(\cdot,t_n), n\in \N \right\}$ is a basis for $H(R_h)$.
$\left\{\underline{R}_h(\cdot,t_n), n\in \N\right\}$ and $\left\{\underline{k}_h(\cdot,f_n), n\in \N\right\}$ stand for the two corresponding orthogonal bases obtained from the Gram-Schmidt process. Then, the linear map ${v_h: H(R_h) \rightarrow H(k_h)}$ can be properly defined:
\begin{equation}\label{eq:defIsomlin}
\forall n\in \N, \quad v_h\left( \underline{R}_h(\cdot,t_n) \right) = \frac{\|\underline{R}_h(\cdot,t_n)\|}{\|\underline{k}_h(\cdot,f_n)\|} \ \underline{k}_h(\cdot,f_n) \ .
\end{equation}
This mapping is an isometry. Enlarging the family $\left\{\underline{R}_h(\cdot,t_n), n\in \N\right\}$ into a complete orthogonal system of $H_h$, $v_h$ is extended to an isometry $u_h$ mapping $H_h$ to $u_h(H_h) = \mathcal{H}_h$. In the following $v_h$ and $u_h$ are both denoted by $u_h$. From now on, $\BB$ will always refer to a fractional Brownian field built from this particular kind of isometry.

\section{Variance of the $h$-increments}\label{sec:hVar}

This section is devoted to proving that the different fractional Brownian fields we can construct satisfy inequalities of the type (\ref{eq:Hincrement}), and in particular the $L^2$-fBf built with the family of isometries we just introduced.

\subsection{The one-dimensional generalised fractional Brownian field}

Firstly, the question is answered in the standard framework of the field indexed over ${(0,1)\times[0,1]}$, and to a larger extent to the corresponding generalised field. Before the main result, we need the following lemma and proposition, whose proofs are reported in Appendix \ref{App:borneKh}.

\begin{lemma}\label{lem:posLF}
For every $n\in \N$, $\underline{R}_h(\cdot, t_n)$ is a positive linear functional, in the sense that for any nonnegative function $g\in H_h$, $\left(\underline{R}_h(\cdot, t_n),g\right)_{H_h}\geq 0$.
\end{lemma}

\noindent Let $W^*_+$ denote the set of positive linear functionals over $W$. As can be seen from equation (\ref{eq:dualite}), any element of $W_+^*$ can also be considered as a finite nonnegative measure.

\begin{proposition}\label{prop:app1}
For all $\eta>0$, there exists a constant $M_{\eta}>0$, such that for all $h_1<h_2 \in (\eta,1/2-\eta)$,  for all $\xi\in W^*_+$,
\begin{equation*}
\int_0^1 \left( K_{h_2}^*\xi (u) - K_{h_1}^*\xi (u) \right)^2 \ {\rm d}u \leq M_{\eta} \left((h_2-h_1) L(h_2-h_1)\right)^2  \|\xi\|_{H_{h_1}^*}^2 \ ,
\end{equation*}
where for all $x\in(-1,1)$, $L(x) = \log(|x|^{-1}) \vee 1$ if $x\neq 0$, and $0$ otherwise.
\end{proposition}

\begin{theorem}\label{prop:varInc1dfBf}
Let $\BB_{h,\xi}$ be the generalised fBf defined in (\ref{eq:1dgfBf}). For all $\eta>0$, there exists a constant $M_{\eta}>0$, such that for all $h_1<h_2 \in (\eta,1/2-\eta)$,  for all $\xi\in W^*$,
\begin{equation*}
\EE\left(\BB_{h_1,\xi} - \BB_{h_2,\xi}\right)^2 \leq M_{\eta} \left((h_2-h_1) L(h_2-h_1)\right)^2  \|\xi\|_{H_{h_1}^*}^2 \ .
\end{equation*}
\end{theorem}

\begin{proof}
If $\xi\in W^*$, then it also belongs to $H_{h_1}^*$. Note that as an element of $H_{h_1}$, the image of $\xi$ by $R_{h_1}$ reads:
\begin{align*}
R_{h_1}\xi &= \sum_{n=0}^{\infty} \frac{(R_{h_1}\xi,\underline{R}_{h_1}(t_n,\cdot))_{H_{h_1}}}{\|\underline{R}_{h_1}(t_n,\cdot)\|_{H_{h_1}}^{2}} \ \underline{R}_{h_1}(t_n,\cdot) \\
&= \sum_{n=0}^{\infty} \frac{\left\{(R_{h_1}\xi,\underline{R}_{h_1}(t_n,\cdot))_{H_{h_1}} \vee 0\right\}}{\|\underline{R}_{h_1}(t_n,\cdot)\|_{H_{h_1}}^{2}} \ \underline{R}_{h_1}(t_n,\cdot) - \sum_{n=0}^{\infty} \frac{\left\{-(R_{h_1}\xi,\underline{R}_{h_1}(t_n,\cdot))_{H_{h_1}} \vee 0\right\}}{\|\underline{R}_{h_1}(t_n,\cdot)\|_{H_{h_1}}^{2}} \ \underline{R}_{h_1}(t_n,\cdot) \\
&= R_{h_1}\xi_+ - R_{h_1}\xi_- \ .
\end{align*}
Since $\|\xi\|_{H_{h_1}^*}^2 = \|\xi_+\|_{H_{h_1}^*}^2 + \|\xi_-\|_{H_{h_1}^*}^2$, it suffices to notice that:
\begin{align*}
\EE\left(\BB_{h_1,\xi_{\pm}} - \BB_{h_2,\xi_{\pm}}\right)^2 = \left\|K_{h_1}^* \xi_{\pm} - K_{h_2}^* \xi_{\pm}\right\|_{L^2}^2 \ ,
\end{align*}
and then to apply Proposition \ref{prop:app1} to $\xi_+$ and $\xi_-$ (which are positive linear functionals of $W^*$, according to Lemma \ref{lem:posLF} and the density of $H_{h_1}$ in $W$).
\end{proof}

While the $W^*$-norm is rougher than a $H_h^*$-norm, it suffices in the following application. Indeed, the next proposition follows by taking $\xi = \delta_t\in W_+^*, t\in[0,1]$ because then $\|\delta_t \|_{H_h^*} \leq C_h \|\delta_t\|_{W^*} = C_h$, where $C_h$ is the norm of the canonical injection from $W^*$ to $H_h^*$.

\begin{corollary}\label{Prop:ContfbfR}
The fractional Brownian field on $(0,1/2)\times [0,1]$ has a continuous version.
\end{corollary} 
 
\begin{proof}
For any $t\in [0,1]$, any $h_1,h_2\in(0,1/2]$,
\begin{equation*}
\EE\left(\BB_{h_1,t} - \BB_{h_2,t}\right)^2 = \EE\left(\BB_{h_1,\delta_t} - \BB_{h_2,\delta_t}\right)^2\ .
\end{equation*}
Proposition \ref{prop:varInc1dfBf} then implies that for any $\eta>0$, there is a constant $M_{\eta}$ such that for all $h_1<h_2\in [\eta,1/2-\eta]$,  and all $t\in[0,1]$,
\begin{align*}
\int_0^1 \left(K_{h_1}^*\delta_t(u) - K_{h_2}^*\delta_t(u)\right)^2 \ {\rm d}u &\leq M_{\eta}\ (h_2-h_1)^2 L(h_2-h_1)^2 \|\delta_t\|_{W^*} \\
&\leq M_{\eta}\ C_{h_1}\ (h_2-h_1)^2 L(h_2-h_1)^2\ .
\end{align*}
It follows that for all $s,t\in[0,1]$,
\begin{align*}
\EE\left(\BB_{h_1,s} - \BB_{h_2,t}\right)^2 &\leq 2\EE\left(\BB_{h_1,s} - \BB_{h_2,s}\right)^2 + 2\EE\left(\BB_{h_2,s} - \BB_{h_2,t}\right)^2 \\ 
&\leq 2 M_{\eta}\ C_{h_1}\ (h_2-h_1)^2 L(h_2-h_1)^2 + 2 |t-s|^{2h_2} \\
&\leq 2 M_{\eta}\ C_{h_1}\ (h_2-h_1)^2 L(h_2-h_1)^2 + 2 |t-s|^{2\eta} .
\end{align*}
The Kolmogorov continuity theorem allows to conclude that the fBf admits a continuous version on $[\eta,1/2-\eta]\times [0,1]$, for any $\eta>0$. The result is proved on $(0,1/2)\times [0,1]$.
\end{proof}

\begin{remark}
Working on $K_h$ for $h\geq 1/2$, we could in fact prove that the fBf has a continuous modification on $(0,1)\times[0,1]$. See for instance \cite{DecreusefondUstunel} where it is proved that $K_h, h\geq 1/2$, has an analytic extension to $(0,1)$.
\end{remark}

\subsection{An application to the mild solutions of a family of stochastic partial differential equations}\label{subsec:spde}

In this section, we suggest an application of the previous results to the solutions of a class of stochastic partial differential equations with additive fractional noises. Our aim is not to solve them explicitly, but rather to prove the $L^2$ continuity of the solutions when the regularity of the (anisotropic) noise varies. The exposition is made on $[0,1]^2$, but extends easily to higher dimensions.

\noindent It was proved in \cite{Carmona} that the tensor product of two abstract Wiener spaces is an abstract Wiener space. This means that for $h_1,h_2\in (0,1)$, $(H_{h_1}\bar{\otimes} H_{h_2}, W\bar{\otimes}_{\varepsilon} W, \mathcal{W}_{h_1}\otimes\mathcal{W}_{h_2})$ is an abstract Wiener space, where $H_{h_1}\bar{\otimes} H_{h_2}$ is the completion of the algebraic tensor product $H_{h_1}\otimes H_{h_2}$ with respect to the norm given by the scalar product: $\forall x_1,x_1'\in H_{h_1}, x_2,x_2'\in H_{h_2}$, $ (x_1\otimes x_2, x_1'\otimes x_2') = (x_1,x_1')_{H_{h_1}} (x_2,x_2')_{H_{h_2}}$; and where $W\bar{\otimes}_{\varepsilon} W$ is the completion of $W\otimes W$ with respect to the norm given by: $\|x\|_{\varepsilon} = \sup\{|x_1^*\otimes x_2^*(x)|: \|x_1^*\|_{W^*}=1, \|x_2^*\|_{W^*}=1\}$. Note that $W\bar{\otimes}_{\varepsilon} W$ is in fact $C_0([0,1]^2)$ with the sup-norm topology (\cite{Ryan} gives a detailed account on topological tensor products), the space of continuous functions vanishing on the axes (this is an application of the Stone-Weierstrass theorem). The canonical operator for this new Wiener space is the tensor operator $R_{h_1}\otimes R_{h_2}$ (in fact, its continuous extension, but we keep the same notation for operators). Let $\mathbb{W}^{h_1,h_2}$ be the white noise associated to the tensor Wiener space. Then, for $(s,t)$ and $(s',t')\in [0,1]^2$,
\begin{align*}
\EE\bigg(\int_{W\bar{\otimes}_\varepsilon W} \langle \delta_{s}\otimes \delta_{t}, w\rangle \ {\rm d}\mathbb{W}^{h_1,h_2}_{w}\ &\int_{W\bar{\otimes}_\varepsilon W} \langle \delta_{s'}\otimes \delta_{t'}, w\rangle \ {\rm d}\mathbb{W}^{h_1,h_2}_{w}\bigg) \\
&= \left(R_{h_1}\otimes R_{h_2}(\delta_{s}\otimes \delta_{t}),R_{h_1}\otimes R_{h_2}(\delta_{s'}\otimes \delta_{t'})\right)_{H_{h_1}\otimes H_{h_2}} \\
&= \left(R_{h_1} \delta_{s},R_{h_1} \delta_{s'}\right)_{H_{h_1}} \ \left(R_{h_2} \delta_{t},R_{h_2} \delta_{t'}\right)_{H_{h_2}} \\
&= R_{h_1}(s,s')\ R_{h_2}(t,t') \ .
\end{align*}
The last expression is the covariance of an anisotropic $(h_1,h_2)$-fractional Brownian sheet (see \cite{Xiao}). Thus, we shall also denote by $\{\mathbb{W}^{h_1,h_2}_{s,t}, \ (s,t)\in [0,1]^2 \}$ this process, and write it as:
\begin{equation}\label{eq:rep1}
\left\{\int_{[0,1]^2} K_{1/2}^*\otimes K_{1/2}^*(\delta_s\otimes \delta_t)(u,v) \ {\rm d}\mathbb{W}_{u,v}^{h_1,h_2} ,\ (s,t)\in[0,1]^2 \right\}\ .
\end{equation}
Similarly, the process
\begin{equation}\label{eq:rep2}
\left\{\int_{[0,1]^2} K_{h_1}^*\otimes K_{h_2}^*(\delta_s \otimes \delta_t)(u,v) \ {\rm d}\mathbb{W}_{u,v} ,\ (s,t)\in[0,1]^2 \right\}\ , 
\end{equation}
where $\mathbb{W}_{u,v}$ is the standard Brownian sheet of $[0,1]^2$ (corresponding to $h_1=h_2=1/2$), is equal in distribution to the two aforementioned processes. This construction still holds with $\xi=\xi_1\otimes\xi_2$ in the space $\textrm{Span}\{\delta_s,s\in[0,1]\}\otimes \textrm{Span}\{\delta_s,s\in[0,1]\}$ and to its completion with respect to the norm $\|K_{h_1}^*\otimes K_{h_2}^*(\xi)\|_{L^2}$, that we denote by $V_{h_1,h_2}$. This corresponds to the standard construction of the Wiener integral with step functions. The image of $V_{h_1,h_2}$ by ${K_{h_1}^*\otimes K_{h_2}^*}$ is denoted by $\mathcal{D}_{h_1,h_2}$ and is the space of integrands of the $(h_1,h_2)$-fractional Brownian sheet. Besides, the processes in (\ref{eq:rep1}) and (\ref{eq:rep2}) extended to $\mathcal{D}_{h_1,h_2}$, are equal in law. We note that the space of integrands of the fractional Brownian sheet on $\R^2$ is partly described in \cite{Leonenko2011}, so that on $[0,1]^2$, we will consider $\tilde{\mathcal{D}}_{h_1,h_2}$ which consists of square integrable functions $\phi$ with support in $[0,1]^2$, for which there is an extension $\tilde{\phi}\in L^2(\R^2)$ with the same support and such that:
\begin{equation*}
\int_{\R^2} |\mathcal{F}\tilde{\phi}(\lambda_1,\lambda_2)|^2 |\lambda_1|^{1-2h_1} |\lambda_2|^{1-2h_2}\ {\rm d}\lambda_1 {\rm d}\lambda_2<\infty \ ,
\end{equation*}
where $\mathcal{F}$ is the Fourier transform. We note that $\tilde{\mathcal{D}}_{h_1,h_2} \subseteq \mathcal{D}_{h_1,h_2}$ and that the equality is not established (in the one-dimensional case, this requires some care, \cite{Jolis}).

This discussion shows that the generalised processes defined by (\ref{eq:rep1}) and (\ref{eq:rep2}) extend to $V_{h_1,h_2}$ and are equal, which can be written as:
\begin{align}\label{eq:equivFracNoise}
\textrm{for any } \xi\in V_{h_1,h_2}, \quad \dot{\mathbb{W}}^{h_1,h_2}\left(K_{1/2}^*\otimes K_{1/2}^* (\xi)\right) \overset{(d)}{=} \dot{\mathbb{W}}\left(K_{h_1}^*\otimes K_{h_2}^* (\xi)\right) \ .
\end{align}

\

Consider now the following family of elliptic SPDEs with additive noise, on a bounded open domain $U\subset [0,1]^2$ with smooth boundary:
\begin{equation}%\label{eq:spde}
\Delta u  = \dot{\mathbb{W}}^{h_1,h_2} \quad \textrm{ on } U, \tag{$\mathcal{L}_{h_1,h_2}$}
\end{equation}
and with the condition that $u=0$ on $\partial U$. It is assumed that all fractional noises below come from a unique white noise $\mathbb{W}$, i.e. that they can be written as in the right-hand term of (\ref{eq:equivFracNoise}). Let $\tilde{\mathcal{D}}_{h_1,h_2}(U)$ be the restriction of $\tilde{\mathcal{D}}_{h_1,h_2}$ to functions with support in $U$. Let $G_U$ be the Green function associated to this problem, which is known to be locally integrable and gives a fundamental solution to the Poisson problem on $U$, for any $\varphi\in L^2(U)$:
\begin{equation*}
\Delta (G_U\ast \varphi) = \varphi, \ \textrm{ where } \  G_U\ast \varphi(x,y) = \int_U G_U\left((x,y),(s,t)\right) \varphi(s,t) \ {\rm d}s {\rm d}t\ , (x,y)\in U\ .
\end{equation*}
This type of equation with the Brownian sheet has already been considered in \cite{NualartTindel} (with a reflection term), and $u$ is a distributional (or mild, as studied in \cite{Dalang}) solution to this problem if it acts on functions $\phi\in C_c^\infty(U)$ in the following way:
\begin{equation*}
\langle u,\Delta \phi\rangle = \int_{[0,1]^2} \phi(x,y)\ {\rm d}\mathbb{W}^{h_1,h_2}_{x,y} \ .
\end{equation*}

\noindent The last integral is a well-defined Wiener integral, since $C_c^\infty(U) \subset \tilde{\mathcal{D}}_{h_1,h_2}(U)$ when $h_1\leq 1/2$ and ${h_2\leq 1/2}$ (which will be assumed from now). Plugging the fundamental solution into the previous equation yields, for $\varphi \in C_c^\infty(U)$:
\begin{equation}\label{eq:mild}
\langle u,\varphi\rangle = \int_{[0,1]^2} G_U\ast \varphi (x,y)\ {\rm d}\mathbb{W}^{h_1,h_2}_{x,y} \ .
\end{equation}

\noindent For this last expression to make sense, we need $G_U\ast \varphi$ to be in $\tilde{\mathcal{D}}_{h_1,h_2}(U)$. This is the case if ${\varphi\in C_c^\infty(U)}$, and $u$ can more generally be defined for any $\varphi$ such that $G_U\ast \varphi \in \tilde{\mathcal{D}}_{h_1,h_2}(U)$. It follows from the definition of $\mathcal{D}_{h_1,h_2}(U)$ and (\ref{eq:equivFracNoise}) that if $G_U\ast \varphi \in \tilde{\mathcal{D}}_{h_1,h_2}(U)$, there exists ${\xi_{\varphi} \in V_{h_1,h_2} (\subseteq C_0([0,1]^2)^*)}$ such that $K_{1/2}^*\otimes K_{1/2}^*(\xi_{\varphi}) = G_U\ast \varphi$. We can now state the following regularity result:
\begin{proposition}
Let $\eta\in(0,1/4)$. Let $(h_1,h_2)$ and $(h_1',h_2')$ be in $(\eta,1/2-\eta)^2$ such that $h_1\leq h_1'$ and $h_2\leq h_2'$, and let $u_{(h_1,h_2)}$ and $u_{(h_1',h_2')}$ be the mild solutions to $(\mathcal{L}_{h_1,h_2})$ and $(\mathcal{L}_{h'_1,h'_2})$ respectively. Then, for all $\varphi$ such that $G_U\ast \varphi \in \tilde{\mathcal{D}}_{h_1,h_2}(U) \cap \tilde{\mathcal{D}}_{h'_1,h'_2}(U)$,
\begin{align*}
\EE\left(u_{(h_1,h_2)}(\varphi) - u_{(h_1',h_2')}(\varphi)\right)^2 \leq M_{\eta}\ (h_1-h_1')^2 L(h_1&-h_1')^2 \|\xi_{\varphi}\|_{H_{h_1}^*\bar{\otimes} H_{h_2}^*}^2 \\
&+ M_{\eta} \ (h_2-h_2')^2 L(h_2-h_2')^2 \|\xi_{\varphi}\|_{H_{h_1}^*\bar{\otimes} H_{h_2}^*}^2
\end{align*}
\end{proposition}

\begin{proof}
Recall first that $\tilde{\mathcal{D}}_{h_1,h_2}(U) \cap \tilde{\mathcal{D}}_{h'_1,h'_2}(U)$ is not empty since it contains $C_c^\infty(U)$, hence let $\varphi$ be such that $G_U\ast \varphi \in \tilde{\mathcal{D}}_{h_1,h_2}(U) \cap \tilde{\mathcal{D}}_{h'_1,h'_2}(U)$. Let $\xi_{\varphi}$ be such that $K_{1/2}^*\otimes K_{1/2}^*(\xi_\varphi) = G_U\ast\varphi$. According to Equations (\ref{eq:equivFracNoise}) and (\ref{eq:mild}), a mild solution of $(\mathcal{L}_{h_1,h_2})$ can be expressed as \[\langle u_{(h_1,h_2)}, \varphi\rangle  = {\int_{[0,1]^2} K_{h_1}^*\otimes K_{h_2}^*(\xi_\varphi)(x,y)\ {\rm d}\mathbb{W}_{x,y}}.\] Thus, the above expectation is in fact the $L^2([0,1]^2)$-norm of $K_{h_1}^*\otimes K_{h_2}^* (\xi_{\varphi}) - K_{h_1'}^*\otimes K_{h_2'}^* (\xi_{\varphi})$. As $\xi_\varphi$ may not have a tensorized form, we express it as the limit of elements of the form: $\sum_{k=1}^n \xi_k\otimes \xi_k' \in C_0([0,1])^*\otimes C_0([0,1])^*$. This reads:
\begin{align*}
\bigg\| K_{h_1}^*\otimes K_{h_2}^*&\left(\sum_{k=1}^n \xi_k\otimes \xi_k'\right) - K_{h'_1}^*\otimes K_{h'_2}^*\left(\sum_{k=1}^n \xi_k\otimes \xi_k'\right) \bigg\|_{L^2([0,1]^2)}^2\\
&\leq 2 \bigg\| K_{h_1}^*\otimes (K_{h_2}^* - K_{h'_2}^*)\left(\sum_{k=1}^n \xi_k\otimes \xi_k'\right)\bigg\|_{L^2}^2
+ 2 \bigg\| K_{h'_2}^*\otimes (K_{h_1}^* - K_{h'_1}^*)\left(\sum_{k=1}^n \xi_k\otimes \xi_k'\right)\bigg\|_{L^2}^2 \ ,
\end{align*}
and up to an orthogonalisation procedure, we can assume that the $\xi_1,\dots, \xi_n$ are orthogonal in $H_{h_1}^*$ (i.e. that $(K_{h_1}^*\xi_i, K_{h_1}^*\xi_j)_{L^2} = \|\xi_i\|_{H_{h_1}^*} \|\xi_j\|_{H_{h_1}^*}\delta_{ij}$) and that the $\xi'_1,\dots, \xi'_n$ are orthogonal in $H_{h'_2}^*$. Then, the tensor product on $L^2([0,1]^2)$ implies that the first term in the above sum decomposes as:
\begin{equation*}
\sum_{k=1}^n \|K_{h_1}^* \xi_k\|_{L^2[0,1]}^2 \ \|(K_{h_2}^*-K_{h'_2}^*)\xi_{k}'\|_{L^2[0,1]}^2 \ ,
\end{equation*}
which is now smaller than:
\begin{equation*}
M_\eta (h_2-h'_2)^2 L(h_2-h'_2)^2 \|K_{h_1}(\cdot,\cdot)\|_{L^2([0,1]^2)} \sum_{k=1}^n \|\xi_k\|_{H_{h_1}^*}^2\ \|\xi_k'\|_{H_{h_2}^*}^2 \ ,
\end{equation*}
using Theorem \ref{prop:varInc1dfBf}. The last sum is exactly $\|\sum_{k=1}^n \xi_k\otimes \xi_k'\|_{H_{h_1}^*\bar{\otimes} H_{h_2}^*}^2$,  which is the result of the Proposition for elements of the algebraic tensor product. So by a density argument, this gives the result for $\xi_\varphi$.
\end{proof}

\subsection{The fractional Brownian field over \texorpdfstring{$\boldsymbol{L^2}$}{L^2}}

The $L^2$-fBf, with a proper family of isometries defined in section \ref{subsec:extended}, is now looked at. A slightly better estimate is attained on the $h$-increments than on the previous results of this section, due to the different underlying structure of the process. In particular, the result of this section would not permit to obtain the previous estimate on solutions of SPDEs.

\begin{theorem}\label{th:regH}
Let $\BB$ be a fBf on $(0,1/2]\times L^2(T,m)$. For any $\eta \in (0,1/4)$ and any compact subset $D$ of $L^2$, there exists a constant $C_{\eta,D}>0$ such that for any $f\in D$, and any $h_1,h_2\in [\eta,1/2-\eta]$,
\begin{equation*}%\label{eq:regH}
\EE \left((\BB_{h_1,f}-\BB_{h_2,f})^2\right) \leq C_{\eta,D}\ (h_2-h_1)^2 \ .
\end{equation*}
\end{theorem}

\begin{proof}%[Proof of Theorem \ref{th:regH}]
This proof is divided into two parts. In the first part, we show that for any $n\in \N$, for any $f\in L^2(m)$, $h\mapsto \underline{k}_h(f,f_n)$ is analytic. This will be needed in the rest of the proof, while in the second part we compute the main estimates. Like $\underline{R}_h$, $\underline{k}_h$ is the Gram-Schmidt transform of $k_h$ : for any $f\in L^2$, $\underline{k}_h(f,f_0) = k_h(f,f_0) \textrm{ and } \forall n\geq 1$,
\begin{align}
\underline{k}_h(f,f_n) &= k_h(f,f_n) - \sum_{j=0}^{n-1}\frac{\left( \underline{k}_h(\cdot,f_j), k_h(\cdot,f_n)\right)_{\mathcal{H}_h}}{\|\underline{k}_h(\cdot,f_j) \|_{\mathcal{H}_h}^2} \underline{k}_h(f,f_j) \label{eq:decompkh1}\\
&= k_h(f,f_n) + \sum_{j=0}^{n-1} \left(\sum_{l=j}^{n-1} \alpha_h(f_n,j,l) \right) k_h(f,f_j) \ , \label{eq:decompkh2}
\end{align}
where the coefficients $\alpha_h(f_n,j,l)$ correspond to the inverse Gram-Schmidt transform. Note that $\alpha_h(f_n,j,l)$ depends on $n$ only through the terms $(\underline{k}_h(\cdot,f_j), k_h(\cdot,f_n))_{\mathcal{H}_h} = \underline{k}_h(f_n,f_j)$, and we define $\alpha_h(g,j,l)$ by an obvious substitution. It is straightforward that for all $f,g\in L^2(m)$, $h\mapsto k_h(f,g)$ is analytic over $h\in(0,1/2)$ (in the sequel we will say, for short, that $k_h(f,g)$ is analytic). Hence, proceding by induction, assume that for all $f,g\in L^2$, $\underline{k}_h(f,f_{n-1})$ and $\alpha_h(g,j,l)$, for $j\leq n-2$ and $j\leq l\leq n-2$, are analytic. We will show that all the terms in (\ref{eq:decompkh1}) are analytic. By the preceding remarks, the choice of $g$ is unimportant since, under the induction hypothesis, $\underline{k}_h(g,f_j)$ is analytic. $k_h(f,f_n)$ is analytic, as was previously stated, so it remains to assess the terms in the sum of (\ref{eq:decompkh1}). For $j\leq n-1$, $(\underline{k}_h(\cdot,f_j), k_h(\cdot,g))_{\mathcal{H}_h} = \underline{k}_h(g,f_j)$, which is analytic by assumption. In particular, this is true for $g=f_n$. Then, decomposing $\underline{k}_h(\cdot,f_{j})$ ($j\leq n-1$) as in (\ref{eq:decompkh2}), $\|\underline{k}_h(\cdot,f_j)\|^2$ is a combination of sums and products of $\alpha_h(g,p,l)$ ($g=f_j$, $p\leq j-1$) and of $k_h(f_i,f_j)$. Hence, it is analytic. The only term left to conclude this induction proof, is $\alpha(g,n-1,n-1)$. The correspondence with (\ref{eq:decompkh1}) indicates that it is equal to $-\underline{k}_h(g,f_{n-1})\ \|\underline{k}_h(\cdot,f_{n-1}) \|_{\mathcal{H}_h}^{-2}$. Again, this is analytic by the induction hypothesis and what we just said on $\|\underline{k}_h(\cdot,f_j)\|^2$.  All this also holds for $R_h$ and the corresponding quantities.

\noindent The analytic property will also be needed for:
\begin{align*}
h'\in (0,1/2] \mapsto  \int_0^1 K_h(t,r) K_{h'}(s,r)\ {\rm d}r \ , 
\end{align*}
for any $h\in(0,1/2], \ s,t\in[0,1]$. In the proof of Lemma 3.1 of \cite{DecreusefondUstunel}, the authors show that for any $s,t\in [0,1]$, $H\in (0,1)\mapsto \int_0^1 K_H(t,r) K_H(s,r)\ {\rm d}r$ is analytic. A direct adaptation of their proof suffices to show what we want.

\vspace{0.2cm}

Turning to the second part of this proof, let $h_1,h_2$ be fixed elements in ${I_{\eta} = [\eta, 1/2-\eta]}$. We recall from (\ref{eq:VarL2fBf}) that: 
\begin{align*}
\EE \left((\BB_{h_1,f}-\BB_{h_2,f})^2\right) &= \int_0^1 \left(K_{h_1}^*R_{h_1}^{-1}u_{h_1}^{-1}k_{h_1}(f,\cdot)-K_{h_2}^*R_{h_2}^{-1}u_{h_2}^{-1}k_{h_2}(f,\cdot)\right)^2(u) \ {\rm d}u \ .
\end{align*}

\noindent From the proof of Lemma \ref{lem:posLF}, we recall that for any $n\in \N$, $\underline{K}_h(t_n,\cdot) = K_h^{-1}\underline{R}_h(t_n,\cdot)$ and that ${\{\underline{K}_h(t_n,\cdot), n\in \N\}}$ is an orthogonal family of $L^2$. The decomposition of $k_h(f,\cdot)$ in $\mathcal{H}_h$ gives:
\begin{align*}
K_h^{-1} u_h^{-1} k_h(f,\cdot) = \sum_{n=0}^{\infty} \underline{k}_h(f,f_n) \frac{\underline{K}_h(t_n,\cdot)}{\|\underline{R}_h(t_n,\cdot)\|_{H_h}} \ ,
\end{align*}
where the equality is in $L^2([0,1])$. By definition of $\underline{K}_h$, $\|\underline{K}_h(t_n,\cdot)\|_{L^2} = \|\underline{R}_h(t_n,\cdot)\|_{H_h}$, so we will drop the last norm in the above formula to consider that ${\{\underline{K}_h(t_n,\cdot), n\in \N\}}$ is an orthonormal family.

\noindent Therefore, 
\begin{equation*}
\EE\left(\BB_{h,f} - \BB_{h',f}\right)^2 = \big\| \sum_{n=0}^{\infty} \underline{k}_h(f,f_n) \underline{K}_h(t_n,\cdot) - \sum_{n=0}^{\infty} \underline{k}_{h'}(f,f_n) \underline{K}_{h'}(t_n,\cdot) \big\|_{L^2}^2 \ .
\end{equation*}
Let us define, for $h,h'$ in ${I_{\eta} = [\eta, 1/2-\eta]}$,
\begin{equation}\label{eq:trunc}
u_N(h,h') = \big\| \sum_{n=0}^N \underline{k}_h(f,f_n) \underline{K}_h(t_n,\cdot) - \sum_{n=0}^N \underline{k}_{h'}(f,f_n) \underline{K}_{h'}(t_n,\cdot) \big\|_{L^2}^2 \ .
\end{equation}
For now, we will assume that this converges uniformly in $h,h'\in I_{\eta}$ and $f\in D$, as $N\rightarrow \infty$. This will be proved in the next paragraph. The limit is denoted by $u(h,h')$ and is the quantity we are interested in. Let us show that $h'\mapsto u_N(h,h')$ is analytic in $h'\in I_{\eta}$, for any $N\in \N$. For this purpose, we rewrite it as:
\begin{align*}
u_N(h,h') = \sum_{n=0}^N \underline{k}_h(f,f_n)^2 + \sum_{n=0}^N \underline{k}_{h'}(f,f_n)^2 - 2\sum_{i=0}^N \sum_{j=0}^N \underline{k}_h(f,f_i) \underline{k}_{h'}(f,f_j) \left(\underline{K}_h(t_i,\cdot),\underline{K}_{h'}(t_j,\cdot)\right)_{L^2} \ .
\end{align*}
The first term is a constant ($N$ and $h$ are fixed), while according to the first part of this proof, the second term is analytic. The coefficients in the linear decomposition of $\underline{K}_{h'}(t_j,\cdot)$ on ${\rm Span}\{K_{h'}(t_l,\cdot), l\leq j\}$ are the one obtained in (\ref{eq:decompkh2}), making the appropriate adaptation to $R_h$. They are also analytic, for the reasons mentioned in the first part, and denoted $\beta_{h'}(j)$, by analogy with the $\alpha_h$'s of the first part. Taking into account the analytic terms $\underline{k}_h(f,f_i) \underline{k}_{h'}(f,f_j)$, we write the double sum in $u_N(h,h')$ in the following way ($\beta$ becomes $\tilde{\beta}$ due to these multiplicative terms):
\begin{align*}
\sum_{i=0}^N \sum_{j=0}^N \tilde{\beta}_h(i) \tilde{\beta}_{h'}(j) \left(K_h(t_i,\cdot),K_{h'}(t_j,\cdot)\right)_{L^2} \ .
\end{align*}
It was proven in the first part that the scalar products are analytic, which finishes to prove our assertion that $u_N$ is analytic in the second variable (and so, in the first variable too when the second is fixed). Now, a standard result on analytic functions states that if a function is the uniform limit on a compact of analytic functions, then it is itself analytic, and the sequence of the derivative functions converges uniformly towards the derivative of the limit (see \cite[p.214]{Rudin}). So, $u(h,h')$ is analytic ($h$ fixed) and its derivative reads:
\begin{align}\label{eq:u'}
u'(h,h') &= \lim_{N\rightarrow \infty} \bigg( 2 \sum_{n=0}^N \underline{k}_{h'}(f,f_n) \underline{k}_{h'}'(f,f_n) - 2\sum_{i=0}^N \sum_{j=0}^N \underline{k}_h(f,f_i) \underline{k}_{h'}'(f,f_j) \left(\underline{K}_h(t_i,\cdot),\underline{K}_{h'}(t_j,\cdot)\right)_{L^2} \nonumber\\
&\quad \quad \quad \quad \quad \quad - 2\sum_{i=0}^N \sum_{j=0}^N \underline{k}_h(f,f_i) \underline{k}_{h'}(f,f_j) \left(\underline{K}_h(t_i,\cdot),\underline{K}_{h'}'(t_j,\cdot)\right)_{L^2} \bigg) \ ,
\end{align}
where the limit is uniform. In fact, it is also uniform in $h$ and $f$, as an adaptation of the proof of \cite{Rudin} (using Cauchy's estimate) shows. The continuity in the first variable of the partial sums $u_N(h,h')$ follows the same line than for the second variable. The continuity in $f\in D$ of these partial sums is obvious from equation (\ref{eq:decompkh2}). As such, a limiting argument implies that $u'(h,h')$ is continuous in both variables and in $f\in D$.

\noindent Hence,
\begin{equation*}
M_u = \sup_{(h,h')\in I_{\eta}^2, f\in D} |u'(h,h')| < \infty \ .
\end{equation*}
We also have that $M^{(2)}_u = \sup_{(h,h')\in I_{\eta}^2, f\in D} |u''(h,h')|$ is finite. For the sake of brevity, we do not develop the proof, which follows by applying the same arguments as we did on the first derivative. Furthermore, we have that $u'(h_1,h_1)=0$. Indeed, the first two terms in (\ref{eq:u'}) annihilates when evaluated at $h_1$, while the last one becomes:
\begin{align*}
\sum_{i\leq j} \underline{k}_{h_1}(f,f_i) \underline{k}_{h_1}(f,f_j) &\left(\left(\underline{K}_{h_1}(t_i,\cdot),\underline{K}_{h_1}'(t_j,\cdot)\right)_{L^2} + \left(\underline{K}_{h_1}(t_j,\cdot),\underline{K}_{h_1}'(t_i,\cdot)\right)_{L^2} \right) \\
&= \sum_{i\leq j} \underline{k}_{h_1}(f,f_i) \underline{k}_{h_1}(f,f_j) \left. \frac{d}{dh} \right|_{h=h_1} \left(\underline{K}_{h}(t_i,\cdot),\underline{K}_{h}(t_j,\cdot)\right)_{L^2} \ ,
\end{align*}
which is zero. Thus, the previous discussion and the mean value theorem applied on the second order Taylor expansion of $u(h_1,h')$ shows that, for $h'\in I_{\eta}$,
\begin{equation}\label{eq:majCut4}
\big\| \sum_{n=0}^\infty \underline{k}_{h_1}(f,f_n) \underline{K}_{h_1}(t_n,\cdot) - \sum_{n=0}^\infty \underline{k}_{h'}(f,f_n) \underline{K}_{h'}(t_n,\cdot) \big\|_{L^2}^2 \leq M^{(2)}_u \ (h_1-h')^2 \ .
\end{equation}

To conclude the proof, it remains to prove the uniform convergence in (\ref{eq:trunc}). We first notice that:
\begin{align*}
\sup_{(h,h')\in I_{\eta}^2, f\in D} | u(h')-u_N(h')| &\leq \sup_{(h,h')\in I_{\eta}^2, f\in D} \big\|\sum_{n=N+1}^\infty \underline{k}_h(f,f_n) \underline{K}_h(t_n,\cdot) - \sum_{n=N+1}^\infty \underline{k}_{h'}(f,f_n) \underline{K}_{h'}(t_n,\cdot)\big\|_{L^2}^2 \\
&\leq 2 \sup_{h\in I_{\eta},f\in D} \sum_{n=N+1}^\infty \underline{k}_h(f,f_n)^2 + 2 \sup_{h'\in I_{\eta}, f\in D} \sum_{n=N+1}^\infty \underline{k}_{h'}(f,f_n)^2 \ .
\end{align*}
The initial problem now comes down to the proof that $k_h(f,f)$ is the uniform limit in ${h\in I_{\eta}}$ and $f\in D$ of $\sum \underline{k}_h(f_n,f)^2$. Let $\nu>0$. We recall that for any $g\in L^2(T)$, $\|k_h(f,\cdot) - k_h(g,\cdot)\|_{\mathcal{H}_h} = \|f-g\|_{L^2}^{2h}$. It follows from the density of $\{f_n\}_{n\in \N}$ in $L^2(T)$ that for any $f\in L^2$, there is an index $\alpha \in \N$ such that ${\|f-f_{\alpha}\|_{L^2} \leq \nu^{1/4\eta}}$. In fact, the compactness of $D$ implies that there is an integer $N_{\nu}$ such that $D$ can be covered by balls of radius at most $\nu^{1/4\eta}$ centered in $\{f_{\alpha_j}, j=1\dots N_{\nu}\} \subset \{f_n\}_{n\in \N}$. Besides, from the construction of $\left\{\underline{k}_h(f_n,\cdot), n\in \N\right\}$, $k_h(f_{\alpha},\cdot) \in \textrm{Span}\left\{\underline{k}_h(f_j,\cdot), j\leq \alpha\right\}$. As a consequence of the previous points, if $f$ is in the ball centered in $f_{\alpha_j}$,
\begin{align*}
\sup_{h\in I_{\eta}} \|k_h(f,\cdot) - k_h(f_{\alpha_j},\cdot)\|_{\mathcal{H}_h}^2 &= \sup_{h\in I_{\eta}} \left(\sum_{n=1}^{\alpha_j} \left(\underline{k}_h(f,f_n) - \underline{k}_h(f_{\alpha_j},f_n)\right)^2 + \sum_{n=\alpha_j+1}^{\infty} \underline{k}_h(f,f_n)^2\right) \\
&= \sup_{h\in I_{\eta}} \|f-f_{\alpha_j}\|_{L^2}^{4h}
\end{align*}
and therefore, $\sup_{h\in I_{\eta}} \sum_{n=\alpha_j+1}^{\infty} \underline{k}_h(f,f_n)^2 \leq \sup_{h\in I_{\eta}} \|f-f_{\alpha_j}\|_{L^2}^{4h}$ which is less than $\nu$. This finally reads: for any $N\geq \alpha = \max_{j=1\dots N_{\nu}} \alpha_j$,
\begin{align*}
\sup_{h\in I_{\eta}, f\in D} |k_h(f,f) - \sum_{n=1}^N \underline{k}_h(f,f_n)^2| &= \sup_{h\in I_{\eta}, f\in D} \sum_{n=N+1}^{\infty} \underline{k}_h(f,f_n)^2 \\
&\leq \nu ,
\end{align*}
so the convergence is uniform and this ends the proof.
\end{proof}

Let $d_m$ denote the distance induced by the $L^2(T,m)$ norm. Expressed in a more general form, the following corollary is obtained from the previous result:
\begin{corollary}\label{cor:regH}
For all compact subset $D$ of $L^2(T,m)$ of $d_m$-diameter smaller than $1$, for any $\eta \in (0,1/4)$, there exists a constant ${C>0}$ (depending on $D$ and $\eta$) such that, $\forall f,f'\in D, \ \forall h,h'\in [\eta,1/2-\eta]$,
\begin{equation}\label{eq:accroissmentsHetF}
\EE \left((\BB_{h,f}-\BB_{h',f'})^2\right) \leq C\ (h'-h)^2 + 2\ m\left(|f-f'|^2\right)^{2(h\wedge h')} \ .
\end{equation}
\end{corollary}

\section{Metric entropy and the fractional Brownian field}

In this section, we address the following question: under which conditions does the fBf have a continuous modification? As this is often the case, the answer is closely related to metric entropy of the indexing collection. Once the answer is made clear, we further study the link between the $h$-fBm and metric entropy, providing an estimate of the small deviations of the process. We remark here that speaking of continuous modification of a process requires the process to have a separable modification, in the sense of Doob. This is always the case for multiparameter processes \cite{davar}, but it is no longer clear when $\R^d$ is replaced by an $L^2$ space. Theorem 2 of \cite[p.153]{Gikhman} provides an answer when the process is indexed by a separable metric space with value in a locally compact space, which includes the $L^2$--fBf.

\subsection{Continuity}

Following equation (\ref{eq:accroissmentsHetF}), on a subdomain $D \subset L^2$ of $d_m$-diameter smaller than $1$:
\begin{align}\label{eq:accroissementsD}
\EE \left((\BB_{h,f}-\BB_{h',g})^{2}\right) &\leq 2C\ \max\left(|h-h'|,m(|f-g|^2)^{1/2}\right)^{4(h\wedge h')} \\
&\leq 2C \ d\left((h,f),(h',g)\right)^{4(h\wedge h')}  \ , \nonumber
\end{align}
where $d$ is the product distance on $(0,1/2]\times L^2(T,m)$.

Let $\tilde{K}$ be a compact of $L^2(T,m)$ of $d_m$-diameter less than $1$, and $a\in (0,1/2)$. Let $\eta>0$, $K_a= [a,1/2-\eta]\times \tilde{K}$ and $C = \left\{\BB_{h,f}, (h,f)\in K_a\right\}$ be a subspace of $L^2(\Omega)$. To measure the distance between points, let $\delta$ be defined by ${\delta(\BB_{\varphi_1},\BB_{\varphi_2}) = \sqrt{\EE(\BB_{\varphi_1}-\BB_{\varphi_2})^2}}$, for $\varphi_1,\varphi_2\in K_a$. For any $\varepsilon<1$, $N(C,\delta,\varepsilon)$ denotes the metric entropy of $C$, that is, the smallest number of $\delta$-balls of radius at most $\varepsilon$ needed to cover $C$. We will also make use of the notation $N(\varepsilon)$ when the context is clear, and denote by $H(\varepsilon)$ the log-entropy $\log\left(N(\varepsilon)\right)$. We give a first result on the modulus of continuity, which is a simple consequence of a famous theorem of Dudley \cite{dudley} and of inequality (\ref{eq:accroissementsD}).

\begin{proposition}\label{prop:modulusCo}
Assume that there exist some $M,\alpha \in \R_+$, such that for all sufficiently small $\varepsilon$, $N(C,\delta,\varepsilon) \leq M \varepsilon^{-\alpha}$. Then, the mapping $x\mapsto x^{2a}\ \sqrt{-\log x}$ is a \emph{uniform modulus of continuity} for $\left\{\BB_{h,f}, (h,f)\in K_a\right\}$, meaning that there exists a measurable $c_{\omega}$ such that almost surely:
\begin{equation*}
\forall (h,f),(h',g) \in K_a, \quad |\BB_{h,f} - \BB_{h',g}| \leq c_{\omega} \ d\left((h,f),(h',g)\right)^{2a} \ \sqrt{-\log d\left((h,f),(h',g)\right)} \ .
\end{equation*}
\end{proposition}

In particular, the fBf on $K_a$ is a.s. Hölder-continuous for any $b<2a$. Such exponential bounds on the entropy appear frequently in statistics, for instance when $C$ is a Vapnick-Cervonenkis class with exponent $\nu$:
\begin{equation*}
\forall \varepsilon>1, \quad N(C,\varepsilon) \leq K \varepsilon^{-2\nu} |\log \varepsilon|^{\nu}.
\end{equation*}
See for instance \cite{adlerTaylor} for a review of these properties. The conditions of the previous Proposition are thus met on a Vapnick-Cervonenkis indexing class, choosing any $\alpha > 2\nu$.

\begin{proof}
The elements of Dudley's Theorem are described as follows: let $L$ be the isonormal process over $L^2(\Omega)$, that is, on the same probability space $\Omega$, the centred Gaussian process whose covariance is given by $\mathbb{E}\left(L(X_1) L(X_2)\right) = \EE\left(X_1 X_2\right)$, for all $X_1,X_2\in C$. Thus ${\mathbb{E}\left((L(X_1)-L(X_2))^2\right)} = {\delta(X_1,X_2)}$. Using a chaining argument and Borel-Cantelli lemma, Dudley proved that $F(x) = \int_0^x \sqrt{\log N(C,\delta,\varepsilon)} \ {\rm d}\varepsilon$ is a modulus of continuity (uniform, and potentially infinite) for the sample paths of $L$ on $C$. A straightforward calculus shows that under the assumptions on the entropy, $x \sqrt{- \log x} \leq F(x) \leq 2x \sqrt{-\log x}$ for all $x \in (0,e^{-1/2}]$. Hence, $x \sqrt{-\log x}$ is a modulus of continuity of $L$. Let $G:x\in \R_+ \mapsto x^{2a}$, so that according to (\ref{eq:accroissementsD}): $\delta(\BB_{h,f},\BB_{h',g}) \leq G\left(d((h,f),(h',g))\right)$. Then, $(h,f)\mapsto L(\BB_{h,f})$ and $(h,f)\mapsto \BB_{h,f}$ have the same law so there exists a measurable subset $\tilde{\Omega}\subseteq \Omega$ of measure $1$, and a measurable $c_{\omega}$ such that for any $\omega \in \tilde{\Omega}$:
\begin{align*}
\forall (h,f),(h',g) \in K_a, \quad |\BB_{h,f} - \BB_{h',g}| &\leq c_{\omega} \ F\left(\delta\left((h,f),(h',g)\right)\right) \\
&\leq c_{\omega} \ F\circ G\left(d\left((h,f),(h',g)\right)\right) \ .
\end{align*}
\end{proof}

The rest of this section is dedicated to improving this result, in various directions. First we argue that studying entropy conditions for the fBf is essentially the same as studying the entropy of the $h$-fBm. This is the purpose of Theorem \ref{th:continuity}, preceded by the following technical lemma. Then, in section \ref{Sec:mpmBm}, we will consider more specific indexing collections for which the regularity results are more precise.

\begin{lemma}\label{lem:CNSconvEnt}
Let $(T_1,d_1)$ and $(T_2,d_2)$ be two compact metric spaces and denote by $d$ the product distance on $T_1\times T_2$. The log-entropies on $(T_1,d_1)$, $(T_2,d_2)$ and $(T_1\times T_2,d)$ are respectively denoted by $H_1(\varepsilon)$, $H_2(\varepsilon)$ and $H(\varepsilon)$. Then, the following lower and upper bounds on $H$ hold (allowing the integral to be infinite):
\begin{align}\label{eq:CNSconvEnt}
\frac{1}{2} \int_0^1 \left(\sqrt{H_1(\varepsilon)} + \sqrt{H_2(\varepsilon)}\right) \ {\rm d}\varepsilon \leq \int_0^1 \sqrt{H(\varepsilon)} \ {\rm d}\varepsilon \leq \sqrt{2} \ \int_0^1 \left(\sqrt{H_1(\varepsilon)}+ \sqrt{H_2(\varepsilon)}\right) \ {\rm d}\varepsilon .
\end{align}
\end{lemma}

\begin{proof}
Let $B^i(c,r)$ the open ball of $(T_i,d_i)$ centred at $c$ with radius $r$, $i=1,2$. Let $\varepsilon>0$ and $\{B^1_j(c^1_j,\varepsilon) \ , \ 1\leq j\leq N_1(\varepsilon)\}$ (resp. $\{B^2_j(c^2_j,\varepsilon) \ , \ 1\leq j\leq N_2(\varepsilon)\}$) be a $\varepsilon$-covering of $T_1$ (resp. $T_2$). First notice that for the product distance $d$, one has $B_i^1(c_i^1,\varepsilon)\times B_j^2(c_j^2,\varepsilon) = B_d((c_i^1,c_j^2),\varepsilon)$ fol all $(i,j)\in \{1,\dots,N_1(\varepsilon)\}\times \{1,N_2(\varepsilon)\}$. A first inequality follows:
\begin{equation*}
N(T_1\times T_2,d,\varepsilon) \leq N_1(\varepsilon)\ N_2(\varepsilon) \ .
\end{equation*}
Reciprocically, if $\{B_1(c_1,\varepsilon),\dots,B_{N(\varepsilon)}(c_{N(\varepsilon)},\varepsilon)\}$ is a $\varepsilon$-covering of $T_1\times T_2$, then each $c_j$ rewrites: $c_j = (c^1_j,c^2_j)$ and so $B_j(c_j,\varepsilon) = B_j^1(c_j^1,\varepsilon)\times B_j^2(c_j^2,\varepsilon)$. Then we have:
\begin{equation*}
T_i \subseteq \bigcup_{j=1}^{N(\varepsilon)} B_j^i(c_j^i,\varepsilon) \ , \quad i=1,2.
\end{equation*}
Hence, $N(\varepsilon) \geq N_1(\varepsilon) \vee N_2(\varepsilon)$. The upper and lower bounds in (\ref{eq:CNSconvEnt}) follow.
\end{proof}

\begin{theorem}\label{th:continuity}
Let $\BB$ be a fBf indexed on a compact subset $I$ of $(0,1/2]$, and $K$ be a compact subset of $L^2(T,m)$ of $d_m$-diameter smaller than $1$. If the Dudley integral converges:
\begin{align}\label{eq:convInt}
\int_{(0,1)} \sqrt{\log N(K,d_m,\varepsilon)} \ d\varepsilon < \infty \ , 
\end{align}
then $\BB$ indexed by $I\times K$ has almost surely continuous sample paths.
\end{theorem}

\begin{remark}
Fernique showed in \cite{Fernique75} that for a stationary process indexed on $\R^d$, the convergence of the Dudley integral is a necessary condition (see \cite[Chap.13]{LedouxTalagrand}, where the result is derived from a majorizing measure argument combined with Haar measures for processes indexed on a locally compact Abelian group). The extension of this result to increment stationary processes is explained clearly in \cite[p.251]{MarcusRosen}. In the case of the $h$-fBm (increment stationary), the indexing collection is an infinite-dimensional Hilbert space, hence it has no locally compact subgroups of noticeable interest. Whether condition (\ref{eq:convInt}) is necessary remains an open problem.
\end{remark}

\begin{proof}
We prove that the convergence of the integral (\ref{eq:convInt}) implies the convergence of this other integral: $\int_{(0,1)} \sqrt{\log N(I\times K,\delta,\varepsilon)} \ d\varepsilon$, which then implies the result according to a famous result of Dudley \cite{dudley}. For $h,h' \in I$, and $f,f'\in K$, let $\iota = \inf I$ which is positive. It readily follows:
\begin{align*}
\delta\left((h,f) , (h',f')\right) \leq C\ d\left((h,f) , (h',f')\right)^{\iota} \ , 
\end{align*}
where $d$ is still the product distance of $d_m$ and $d_1$ (the absolute value distance on $\R$). Since $\int_0^1 \left(\log(N(I,d_1,\varepsilon)) \right)^{1/2} \ {\rm d}\varepsilon < \infty$, Lemma \ref{lem:CNSconvEnt} implies that the convergence of the Dudley integral for $d$ is equivalent to the convergence of the Dudley integral for $d_m$. Hence the result.
\end{proof}

Examples of indexing classes for which the fBf is a.s. continuous will be discussed in section \ref{Sec:applications}. We simply recall that an object as simple as the Brownian motion indexed over the Borel sets of $[0,1]^2$, that is, the centred Gaussian process with covariance:
\begin{equation*}
\forall U,V\in \mathcal{B}([0,1]^2), \quad \EE\left(W_U \ W_V\right) = \lambda(U\cap V) \ ,
\end{equation*}
is almost surely unbounded \cite[p.28]{adlerTaylor}.

\subsection{Small deviations}

In this paragraph, we explore with more details the relationship with metric entropy. In Theorem \ref{theo:smallBall}, a connection between the small deviations of the $h$-fBm and metric entropy is expressed, opening the field of the measure of local properties of the fBf, such as Chung laws of the iterated logarithm or measure of Hausdorff dimension of the paths.

Perhaps the most general result on entropy and small ball probabilities over Wiener spaces is due to Goodman \cite{Goodman} who showed that, for $K_{\mu}$ the unit ball of the RKHS of $\mu$,
\begin{equation*}
\lim_{\varepsilon \rightarrow 0} \ \varepsilon^2 H(K_{\mu}, \varepsilon) = 0 \ ,
\end{equation*}
where $H(K_{\mu},\varepsilon)$ is the log-entropy computed under the Banach norm (which makes $K_{\mu}$ compact in the Banach space).

Kuelbs and Li \cite{KuelbsLi} considerably refined this equality, establishing a link between the small balls of a Gaussian measure and the metric entropy of $K_{\mu}$. To state it, let us introduce some notation: as $x\rightarrow a$, we shall write $f(x)\approx g(x)$ if
\begin{align*}
0<\liminf_{x\rightarrow a} f(x)/g(x) \leq \limsup_{x\rightarrow a} f(x)/g(x) < \infty \ .
\end{align*}
Let us assume that there exists a function $f$ which is regularly varying at infinity\footnote{i.e., a function such that there exists $\rho \in \R$ satisfying $\lim_{x\rightarrow \infty} f(\lambda x)/f(x) = \lambda^{\rho}$ for all $\lambda>0$.} and two constants $c_1$ and $c_2$ such that: \[c_1 \ f(\varepsilon) \leq -\log \mu(B(0,\varepsilon)) \leq c_2 \ f(\varepsilon) \ , \] as $\varepsilon\rightarrow 0$. Since $f$ is regularly varying, there exists $\alpha\geq 0$ and a slowly varying function at infinity\footnote{a regularly varying function with $\rho=0$.} $J$ such that $f(\varepsilon) = \varepsilon^{-\alpha} J(\varepsilon^{-1})$. If $\alpha>0$, then:
\begin{equation*}
H(K_{\mu},\varepsilon) \approx \varepsilon^{-2\alpha/(2+\alpha)} J\left(\varepsilon^{-1}\right)^{2/(2+\alpha)} .
\end{equation*}

\noindent In some circumstances this yields precise results. This is the case for fBm: let $\BB^h$ a fBm over $[0,1]$ and notice that for any $\varepsilon>0$, $N([0,1],d_h, \varepsilon) = \varepsilon^{-1/h}$ where $d_h$ is the distance on $[0,1]$ induced by the Lévy fBm. Then:
\begin{equation}\label{eq:smallBallsmultiparam}
-\log \PP \left( \sup_{t\in [0,1]} |\BB^h_t| \leq \varepsilon \right) \approx \varepsilon^{-1/h}, \quad \forall \ 0<\varepsilon<1 .
\end{equation}
Historically, this was first proved in \cite{Monrad} for one-parameter processes, then extended independently in \cite{ShaoWang,Talagrand95} to the multiparameter setting. We will generalise this result for the $h$-fBm. The difference between our result and the result of \cite{KuelbsLi} is then discussed in Remark \ref{Rem:smallBalls}.

In order to extend the result of equation (\ref{eq:smallBallsmultiparam}), the following lemma will be needed. It is interesting in itself, since it establishes that for each $h\in (0,1/2)$, the $h$-fBm is \emph{strongly locally nondeterministic} (SLND) in the following sense:

\begin{lemma}\label{lem:SLND}
Let $h\in (0,1/2)$. There exists a positive constant $C_0$ such that for all $f\in L^2(T,m)$ and for all $r\leq \|f\|$, the following holds:
\begin{equation*}
{\rm Var}\left(\BB^h_f \ | \ \BB^h_g, \|f-g\|\geq r\right) = C_0 r^{2h}.
\end{equation*}
\end{lemma}

\begin{proof}
The proof for $C_0\geq 0$ relies essentially on the metric structure of the covariance of fBm, from which follow increment stationarity and scale invariance of the process. As such, the proof is the same as for the Lévy fractional Brownian motion in $\R^d$, as it appeared first in Lemma 7.1 of \cite{Pitt}. In his paper, Pitt used Fourier analysis to obtain $C_0>0$. Although this tool is finite-dimensional (because the Lebesgue measure is) and despite the non-existence of a standard infinite-dimensional Fourier transform, Gaussian measures provide a natural extension to an infinite-dimensional framework. Using Pitt's arguments, we obtain that for $\tilde{f} = (r/\|f\|) \ f$:
\begin{equation*}
{\rm Var}\left(\BB^h_f \ | \ \BB^h_g, \|f-g\|\geq r\right) = {\rm Var}\left(\BB^h_{\tilde{f}} \ | \ \BB^h_g, \|\tilde{f}-g\|\geq r=\|\tilde{f}\|\right) = C_0 r^{2h}.
\end{equation*}
If $C_0$ was to be $0$, there would exist a sequence of random variables $B_n$ of the form ${B_n = \sum_j a_j(n) \BB^h_{g_j}}$, where $\|g_j-\tilde{f}\|\geq \|\tilde{f}\|$, such that $B_n$ converges to $\BB^h_{\tilde{f}}$ in $L^2(\PP)$. Stated differently, the sequence $b_n^* = \sum_j a_j(n)\langle \tilde{\mathcal{K}}_h k_h(g_j,\cdot), \cdot \rangle$ converges to $b^* = \langle \tilde{\mathcal{K}}_h k_h(\tilde{f},\cdot), \cdot \rangle$ in $L^2(\mu)$.
For $\varphi \in L^2(\mu)$, define:
\begin{equation*}
\mathcal{F}\varphi(x^*) = \int_E \cos\langle x^*,x \rangle \ \varphi(x) \ {\rm d}\mu(x) \ , \ x^*\in E^* .
\end{equation*}
The main property of $\mathcal{F}$ that we will need is the following: if $\varphi \in E^*$, then $\mathcal{F}\varphi(x^*)\neq 0$ if and only if $\varphi = \lambda.x^*$, for some $\lambda \in \R\setminus \{0\}$. The proof, which is mainly calculus, is postponed to Appendix \ref{App:Fourier}. Let $\mathcal{F}_0$ be a restriction of $\mathcal{F}$ satisfying:
\begin{align*}
\mathcal{F}_0\varphi(f_2) = \int_E \cos\left(\langle \tilde{\mathcal{K}}_h k_h(f_2,\cdot),x \rangle\right) \ \varphi(x) \ {\rm d}\mu(x) \ , \ f_2\in L^2(T), \varphi \in L^2(\mu) .
\end{align*}
For any fixed $f_1\in L^2(T)$, $\mathcal{F}_0 \left(\tilde{\mathcal{K}}_h k_h(f_1,\cdot) \right)$ is non-zero only if $f_2\in L^2(T)$ is such that, for some $\lambda \in \R\setminus \{0\}$, ${\tilde{\mathcal{K}}_h k_h(f_2,\cdot) = \lambda \tilde{\mathcal{K}}_h k_h(f_1,\cdot)}$. Hence it is non-zero only if $k_h(f_2,\cdot) = \lambda k_h(f_1,\cdot)$. A by-product of the proof of Lemma \ref{lem:LinInd} is that this equality can only hold if $f_1 = f_2$. This implies that the support of $\mathcal{F}_0 b_n^*$ is included in $\{g_j \ , \ j\in \N\}$ which is strictly disjoint from the support of $\mathcal{F}_0b^*$. 

Applying the Cauchy-Schwarz inequality to $|\mathcal{F}_0 b_n^*(f) - \mathcal{F}_0 b^*(f) |$, one proves that for all $f\in L^2(T)$, $\mathcal{F}_0 b_n^*(f) \rightarrow \mathcal{F}_0 b^*(f)$ as $n$ tends to infinity. This is a contradiction with the fact that the supports are strictly disjoint.
\end{proof}

\begin{theorem}\label{theo:smallBall}
Let $h\in (0,1/2)$, $\BB^h$ a $h$-fractional Brownian motion and $K$ a compact set in $L^2(T,m)$. Then, for some constant $k_1>0$,
\begin{equation*}
\PP \left(\sup_{f \in K}|\BB^h_f|\leq \varepsilon \right) \leq \exp\left(-k_1 \ N(K,d_h,\varepsilon)\right) \ ,
\end{equation*}
and if there exists $\psi$ such that for any $\varepsilon>0$, $N(K,d_h,\varepsilon)\leq \psi(\varepsilon)$ and ${\psi(\varepsilon)\approx \psi(\varepsilon/2)}$, then for some constant $k_2>0$,
\begin{equation*}
\PP \left(\sup_{f \in K}|\BB^h_f|\leq \varepsilon \right) \geq \exp\left( -k_2 \ \psi(\varepsilon) \right) \ . 
\end{equation*}

\end{theorem}

\begin{proof}
The lower bound follows from Lemma 2.2 in \cite{Talagrand95} and is a general result for Gaussian processes. The upper bound is specific to Lévy-type fractional Brownian motions and is a consequence of the SLND property proved above, and of an argument of conditional expectations as described in \cite{Monrad}.

Let $\eta>0$ and $M(\eta)\subset K$ be a finite set of maximal cardinality, in the sense that for any elements, $f\neq g \in M(\eta) \Rightarrow \|f-g\|\geq \eta^{1/2h}$. The cardinal $|M(\eta)|$ is generally referred to as packing number. The elements of $M(\eta)$ are arbitrarily ordered and denoted $f_1,\dots, f_{|M(\eta)|}$. Then,
\begin{align*}
\PP \left( \sup_{f\in K} |\BB^h_f|\leq \varepsilon \right) \leq \PP\left( \sup_{f\in M(\eta)} |\BB^h_f|\leq \varepsilon \right) \ ,
\end{align*}
and since the conditional distributions of a Gaussian process are Gaussian, the SLND property of Lemma \ref{lem:SLND} implies that for any $k\in \{2,\dots,|M(\eta)|\}$:
\begin{align*}
\PP\left(|\BB^h_{f_k}|\leq \varepsilon \ \big| \ \BB^h_{f_j} \ , \ j\leq k-1 \right) = \Phi(C_0^{-1}\eta^{-1}\varepsilon) \ ,
\end{align*}
where $\Phi$ is the cumulative distribution function of a standard normal random variable. By repeated conditioning,
\begin{align*}
\PP\left( \sup_{f\in M(\eta)} |\BB^h_f|\leq \varepsilon \right) \leq \left(\Phi(C_0^{-1}\eta^{-1}\varepsilon)\right)^{|M(\eta)|} \ .
\end{align*}

\noindent As $N(2\varepsilon) \leq |M(\varepsilon)|$, taking $\eta = \varepsilon/2$ in the previous inequality yields:
\begin{align*}
\PP \left( \sup_{f\in K} |\BB^h_f|\leq \varepsilon \right) \leq \exp(-k_1 N(\varepsilon)) \ ,
\end{align*}
with $k_1 = -\log\Phi(2C_0^{-1})>0$.
\end{proof}

Estimating the small balls of the fBf (i.e. when $h$ is not fixed anymore) seems more complicated. Talagrand's lower bound estimate still holds, leading to: for $K$ compact in $(0,1/2)\times L^2(T,m)$, $\PP(\sup_{(h,f)\in K} |\BB_{h,f}|\leq \varepsilon) \geq \exp(-k_2\ N(K,d_{\BB},\varepsilon))$. A sharp estimate of this last entropy in terms of the entropy on both coordinates would be required, while for the upper bound, the notion of SLND for the fBf does not seem appropriate: intuitively, the regularity in the $h$-direction contrasts with the nondeterminism studied above. 

\begin{remark}\label{Rem:smallBalls}
Theorem \ref{theo:smallBall} is rather different from what is obtained via \cite{KuelbsLi}. Indeed, their result is concerned with the supremum of the process over the elements of the RKHS measured with the Banach norm:
\begin{align*}
\PP\left(\sup_{\varphi \in B_{E_h^*}(0,1)}|\BB^h_{\varphi}|\leq \varepsilon\right) = \PP\left(\|\BB^h\|_{E_h}\leq \varepsilon\right) = \mu_h\left(B_{E_h}(0,\varepsilon)\right)
\end{align*}
and $\mu_h\left(B_{E_h}(0,\varepsilon)\right) = \mathcal{W}_h\left(B_W(0,\varepsilon)\right)$ by isometry. Meanwhile, it comes from (\ref{eq:smallBallsmultiparam}) that:
\begin{align*}
\mathcal{W}_h\left(B_W(0,\varepsilon)\right) &= \PP\left(\sup_{t\in [0,1]}|\BB^h_t|\leq \varepsilon\right) \approx \exp\left\{-\frac{1}{\varepsilon^{1/h}}\right\} \ ,
\end{align*}
which in general is different from our bound. This does not contradict the previous Theorem, because of the difference between the Hilbert norm and the Banach norm.
\end{remark}

\section{Applications to the regularity of the multiparameter and set-indexed fractional Brownian fields}\label{Sec:applications}

In this section, we present the fBf and the $h$-fBm in the more familiar framework of multiparameter processes, enhancing the fact that these processes are rather different from the Lévy fractional Brownian motion and the fractional Brownian sheet, as well as from their multifractional counterparts (\cite{herbin}). This study is then extended to set-indexed processes. In both cases, the meeting point will be that the fBf is now considered as a multifractional process, meaning that on the indexing collection $\A$ (to be specified), we have a function $\h:\A \rightarrow (0,1/2]$, and denote $\BB^{\h}$ the process indexed over $\A$ defined by $\left\{\BB_{\h(U),U}, U\in \A \right\}$. This framework allows to establish more precise regularity results, such as the measure of local Hölder exponents.

\subsection{Multiparameter multifractional Brownian motion}\label{Sec:mpmBm}

For some $d\in \N^*$, let $\A = \{[0,t], t\in [0,1]^d\}$. Let $\BB$ the fBf on $L^2([0,1]^d,m)$ where $m$ is not necessarily the Lebesgue measure. Then a multiparameter multifractional Brownian motion is a process $\BB^{\h}$ defined for some function $\h:[0,1]^d\rightarrow (0,1/2]$ by:
\begin{align*}
\forall t\in [0,1]^d , \quad \BB^{\h}_t = \BB_{\h_t,\mathbf{1}_{[0,t]}} \ .
\end{align*}
For $\h$ a constant function equal to $1/2$, this is the usual Brownian sheet of $\R^d$ (when $m=\lambda_d$). For any other constant function, this is neither the fractional Brownian sheet nor the Lévy fractional Brownian motion, but a process called multiparameter fBm (mpfBm) with covariance:
\begin{align*}
\EE\left(\BB^{\h}_s \ \BB^{\h}_t\right) = \frac{1}{2} \left(m([0,s])^{2\h} + m([0,t])^{2\h} - m([0,s]\bigtriangleup [0,t])^{2\h}\right) \ ,
\end{align*}
where $\bigtriangleup$ is the symmetric difference between sets. Some of the differences between this process and the aforementioned are discussed in \cite{ehem}. As stated in the introduction, one can also obtain the Lévy fractional Brownian motion from the fBf, choosing another class $\A$ and a specific measure. Hence, the results of regularity for the Lévy fBm (or its multifractional counterpart, see for instance \cite{herbin}) follow from the results of the next subsection rather than this one.

In this case and unlike the previous section, the entropy of $\A$ is perfectly known when $m$ is the Lebesgue measure\footnote{which is assumed until the end of this section.}, and the Dudley integral is easily seen to be finite. Hence the fBf on $\A$ has a continuous modification, and so does any multiparameter mBm. It is possible to establish precise Hölder regularity coefficients. For easier comparison with prior works on mpfBm, we will not use the distance $d_m$ but a variant defined as:
\begin{align*}
\textrm{for } s,t\in [0,1]^d,\quad d'_m(s,t) = m([0,s]\bigtriangleup [0,t]) \ .
\end{align*}
Note that $d'_m(s,t) = d_m(\mathbf{1}_{[0,s]},\mathbf{1}_{[0,t]})^2$. A result of \cite{ehar} states that this distance is equivalent to the Euclidean distance when $m$ is the Lebesgue measure and the set of indexing points stays within a compact away from $0$.

For a stochastic process $X$ indexed on $\A$, let us define the deterministic pointwise H\"older exponent at $t_0\in \A$:
\begin{equation*}
\mathbb{\bbalpha}_X(t_0)=\sup\left\{ \alpha:\;\limsup_{\rho\rightarrow 0}
\sup_{s,t\in B_{d'_m}(t_0,\rho)}
\frac{\EE\left(|X_s-X_t|^2\right)}{\rho^{2\alpha}}<\infty \right\} \ ,
\end{equation*}
where $B_{d'_m}(t_0,\rho)$ is the ball of the $d'_m$ distance. Similarly, the deterministic local H\"older exponent is:
\begin{equation*}
\widetilde{\mathbb{\bbalpha}}_X(t_0)=\sup\left\{ \alpha:\;\limsup_{\rho\rightarrow 0}
\sup_{s,t\in B_{d'_m}(t_0,\rho)}
\frac{\EE\left(|X_s-X_t|^2\right)}{d'_m(s,t)^{2\alpha}}<\infty \right\}.
\end{equation*}
We will compare these exponents to their stochastic analogue, straightforwardly defined getting rid of the expectation in the above definitions. The random coefficients are denoted $\alphar_X(t_0)$ and $\widetilde{\alphar}_X(t_0)$. This extends to continuous (deterministic) functions on $\A$.

As the terminology is commonly accepted in the multifractional literature, a \emph{regular} multiparameter mBm will be a fBf with a function $\h$ such that, at each point, the value of the function is smaller than its local and pointwise exponents (ie $\h_t\leq \alphar_{\h}(t)$).

\begin{proposition}
Let $\BB^{\h}$ be a regular multiparameter mBm on $[0,1]^d$. Then, for all $t_0\in [0,1]^d$, both equalities hold almost surely:
\begin{equation*}
\alphar_{\BB^{\h}}(t_0)= \h_{t_0} \ \textrm{ and } \ \widetilde{\alphar}_{\BB^{\h}}(t_0)= \h_{t_0} \ .
\end{equation*}
\end{proposition}
\noindent When $t_0\neq 0$, these equalities still hold true for the exponents defined replacing $d'_m$ with the Euclidean distance. This is another consequence of the equivalence between those distances on a compact away from $0$.

\begin{proof}
The first step is to evaluate $\mathbb{\bbalpha}_{\BB^{\h}}(t_0)$ and $\widetilde{\mathbb{\bbalpha}}_{\BB^{\h}}(t_0)$. A result of \cite{ehar} then states that a Gaussian process $X$, indexed by a collection of sets satisfying certain technical assumptions has the following property:
\begin{equation}\label{eq:HolderExp}
\PP\left( \alphar_{X}(t_0) = \mathbb{\bbalpha}_X(t_0)\right) = 1 \ \textrm{ and } \ \PP\left(\widetilde{\alphar}_{X}(t_0) = \widetilde{\mathbb{\bbalpha}}_X(t_0)\right) = 1 \ .
\end{equation}
As discussed in the aforementioned paper, the technical assumptions are satisfied by the class $\A$ of rectangles. A result of that sort actually originated in \cite{ehjlv}, but we use the one in \cite{ehar} to introduce the extended results of the following section on set-indexed processes.

Let $K$ be a compact of $[0,1]^d$ with $d'_m$-diameter smaller than $1$, whose interior contains $t_0$. As a consequence of the continuity of $\h$, ${\h(K) \subseteq [\eta, 1/2-\eta]}$ for some $\eta>0$. For all $s,t\in B(\rho)$, the ball centred in $t_0$ of radius $\rho$, 
\begin{align*}
\EE\left(\BB^{\h}_t - \BB^{\h}_s\right)^2 &\geq \frac{1}{2} \EE\left(\BB_{\h_t,t} - \BB_{\h_t,s}\right)^2 - \EE\left(\BB_{\h_t,s} - \BB_{\h_s,s}\right)^2 \\
&\geq \frac{1}{2} d'_m(s,t)^{2\h_t} - C_{\eta,K} (\h_t-\h_s)^2  \ ,
\end{align*}
where we used Theorem \ref{th:regH}, and this inequality yields that for any $\alpha > \inf_{B(\rho)} \h$: 
\begin{align*}
\frac{\EE\left(\BB^{\h}_t - \BB^{\h}_s\right)^2}{d'_m(s,t)^{2\alpha}} \geq \frac{1}{2} d'_m(s,t)^{2\inf_B \h - 2\alpha} - \tilde{C}_{\eta,K} d'_m(s,t)^{2\inf_B \alphar_{\h}- 2\alpha} \ .
\end{align*}
The regularity property of $\h$ implies that for $\rho$ sufficiently small, $\alpha$ can be chosen so that $\inf_{B(\rho)} \alphar_{\h} \geq \alpha > \inf_{B(\rho)} \h$, and the previous inequality diverges as $\rho\rightarrow 0$. Hence \[\mathbb{\bbalpha}_{\BB^{\h}}(t_0) \leq \lim_{\rho\rightarrow 0} \inf_{B(\rho)} \h = \h(t_0) \ .\]

The converse inequality follows from the result of Corollary \ref{cor:regH} and the same reasoning. Thus the deterministic exponents are both equal to $\h_{t_0}$, and the property (\ref{eq:HolderExp}) leads to the result.
\end{proof}

This result, which holds for all points, almost surely, is greatly strenghtened into paths properties by the following proposition:

\begin{proposition}
Let $\BB^{\h}$ be a regular multiparameter mBm on $[0,1]^d$. Then, almost surely,
\begin{align*}
\forall t_0\in[0,1]^d, \quad \widetilde{\alphar}_{\BB^{\h}}(t_0) = \h_{t_0} \ \textrm{ and } \ \alphar_{\BB^{\h}}(t_0) \geq \h_{t_0} \ .
\end{align*}
\end{proposition}

\begin{proof}
This is a direct application of Theorem 5.4 of \cite{ehar} and of the values of $\mathbb{\bbalpha}_{\BB^{\h}}$ and $\widetilde{\mathbb{\bbalpha}}_{\BB^{\h}}$ computed in the proof of the previous proposition.
\end{proof}

\vspace{0.3cm}

In the case of the SIfBm (\cite{ehar}) or of the regular mBm (\cite{ehjlv}), the previous uniform lower bound of the pointwise exponent is an equality. This provides tangible argument for an improvement of our result, but the question is left open for now. Finally, we mention \cite{HerbinXiao} where a LND property of the mpfBm is exhibited with a different argument than ours, yielding geometric results on the sample paths of the process, as well as a Chung law.

\subsection{Set-Indexed multifractional Brownian motion}

This section is a discussion on a natural extension of the results on multiparameter processes to a wider class of indexing collections. The framework of set-indexed processes of Ivanoff and Merzbach \cite{Ivanoff}, and the results of \cite{ehar} provide a coherent definition of Hölder exponents and fine regularity results.

Let $T$  be a  locally compact complete separable metric and measure space with metric $d$ and Radon measure $m$ defined on the Borel sets of $T$.

\begin{definition}\label{basic}
A nonempty class $\mathcal{A}$ of compact, connected subsets of $T$ is called an {\em indexing collection} if it satisfies the following:
\begin{enumerate}
 \item $\emptyset \in \mathcal{A}$, and the interior $A^{\circ}\neq A$ if $A\neq \emptyset$ or $T$. In addition, there is an increasing sequence $(B_n)_{n\in\N}$ of union of sets of $\A$ such that $T = \cup_{n=1}^{\infty} B_n^{\circ}$.
\item $\mathcal{A}$ is closed under arbitrary intersections and if $A,B\in \mathcal{A}$ are nonempty, then $A\cap B$ is nonempty. 
The $\sigma$-algebra generated by $\mathcal{A}$is equal to $\mathcal{B}$, the collection of Borel sets of $T$.
\item {\em [Separability from above]}, there exists a nested sequence of finite dissecting classes $\A_n$ whose elements approximate sets $A\in\A$ from above (they are bigger for the inclusion), and this approximation is finer as $n\in \N$ increases, until it equals $A$ at the limit. This is fully and precisely described in \cite{Ivanoff}.
\end{enumerate}

\end{definition}

The construction of a set-indexed multifractional Brownian motion relies on what was said at the beginning of Section \ref{Sec:applications}, and it follows that for all $U\in\A, \ \BB^{\h}_U = \BB_{\h_{U},\mathbf{1}_U}$ is a well defined set-indexed process. Its multiple Hölder coefficents are defined identically than in the multiparameter case (think of point $t_0$ as a set $[0,t_0]\in \A$ in the previous paragraph), with respect to the distance $d_m'$, defined for all $U,V\in \A$, by $d_m'(U,V) = m(U\bigtriangleup V)$.

Under entropic assumptions ensuring the convergence of Dudley's integral for $d_m'$ (as described in \cite{ehar}, where the authors consider entropy for the inclusion, reinforcing the \emph{separability from above} condition of $\A$), the results on local and pointwise Hölder exponents, as presented for multiparameter processes, also hold for the SImBm:
\begin{proposition}
Let $\A$ be an indexing collection satisfying \emph{Assumption 1} of \cite{ehar}. Let $\BB^{\h}$ be a regular SImBm on $\A$. Then, for all $U_0\in \A$, both equalities hold almost surely:
\begin{equation*}
\widetilde{\alphar}_{\BB^{\h}}(U_0)= \h_{U_0} \ \textrm{ and } \ \alphar_{\BB^{\h}}(U_0)= \h_{U_0} \ .
\end{equation*}
\end{proposition}

\begin{proposition}
Let $\A$ be an indexing collection satisfying \emph{Assumption 1} of \cite{ehar}. Let $\BB^{\h}$ be a regular SImBm on $\A$. Then, almost surely,
\begin{align*}
\forall U_0\in\A, \quad \widetilde{\alphar}_{\BB^{\h}}(U_0) = \h_{U_0} \ \textrm{ and } \ \alphar_{\BB^{\h}}(U_0) \geq \h_{U_0} \ .
\end{align*}
\end{proposition}

Finally, we note that there is no evidence of another (Gaussian) process with prescribed regularity in a general set-indexed setting, other than the one we defined. In particular, the (multi-)fractional Brownian sheet and Lévy fBm do not have extensions in the set-indexed setting.

\vspace{0.5cm}

\appendixpage
\begin{appendices}
\section{A bound for the increments of $K_h^*$ in $L^2$}\label{App:borneKh}

This first appendix collects the proofs of the technical results of the beginning of section \ref{sec:hVar}.
We recall that $W^*_+$ denote the set of positive linear functionals over $W$, and that for all $x\in(-1,1)$, $L(x) = \log(|x|^{-1}) \vee 1$ if $x\neq 0$, and $0$ otherwise.

\begin{proof}[Proof of Lemma \ref{lem:posLF}]
For any $t\in \mathbb{D}$, we write $\underline{K}_h(t_n,\cdot) = K_h^{-1}\underline{R}_h(\cdot,t_n)$. $\underline{R}_h$ is the Gram-Schmidt transform of $R_h$, which implies that $\underline{R}_h(\cdot,t_0) = R_h(\cdot,t_0) \textrm{ and } \forall n\geq 1$,
\begin{align}\label{eq:GSRh}
\underline{R}_h(\cdot,t_n) = R_h(\cdot,t_n) - \sum_{j=0}^{n-1}\frac{\left( \underline{R}_h(\cdot,t_j), R_h(\cdot,t_n)\right)_{H_h}}{\|\underline{R}_h(\cdot,t_j) \|^2} \underline{R}_h(\cdot,t_j) \ .
\end{align}
Hence, $\underline{K}_h$ can be written:
\begin{align}\label{eq:K_h-GS}
\underline{K}_h(t_n,\cdot) &= K_h(t_n,\cdot) - \sum_{j=0}^{n-1} \frac{\left( \underline{R}_h(\cdot,t_j), R_h(\cdot,t_n)\right)_{H_h}}{\|\underline{R}_h(\cdot,t_j)\|^2} \underline{K}_h(t_j,\cdot) \nonumber\\
&= K_h(t_n,\cdot) - \sum_{j=0}^{n-1} \frac{\left( \underline{K}_h(t_j,\cdot), K_h(t_n,\cdot)\right)_{L^2}}{\|\underline{K}_h(t_j,\cdot)\|_{L^2}^2} \underline{K}_h(t_j,\cdot) \ ,
\end{align}
and this shows that $\{\underline{K}_h(t_n,\cdot), n\in \N\}$ is the Gram-Schmidt orthogonal family of $L^2$, obtained from $\{K_h(t_n,\cdot), n\in \N\}$. Then for any $g\in L^2$ such that $g \geq 0$, the non-negativeness of $K_h(t,s), \forall t,s\in[0,1]$ (see the closed form (\ref{eq:K_h})), implies that $\int_0^1 g(s) K_h(t_n,s) \ {\rm d}s\geq 0$. Thus, if $g$ is orthogonal to the linear span of $\left\{\underline{K}_h(t_0,\cdot), \dots, \underline{K}_h(t_{n-1},\cdot)\right\}$, it follows from (\ref{eq:K_h-GS}) that $\int_0^1 g \ \underline{K}_h(t_n, \cdot)$ is non-negative. It is obviously also the case if $g\in {\rm Span}\left\{\underline{K}_h(t_0,\cdot), \dots, \underline{K}_h(t_{n-1},\cdot)\right\}$, hence $\left(\underline{K}_h(t_n,\cdot), \ \cdot\ \right)_{L^2}$ is a positive linear functional over ${\rm Span}\left\{\underline{K}_h(t_j,\cdot), j\in\N\right\}$. This leads to the following partial result:
\begin{equation*}
\textrm{for any } j\in \N \ , \quad \underline{R}_h(t_j,t_n) = \left(\underline{K}_h(t_n,\cdot), K_h(t_j,\cdot) \right)_{L^2} \geq 0 \ .
\end{equation*}
Now let $g\in H(R_h)$ such that $g\geq 0$. As any element of $H(R_h)$, $g$ can be approximated by a sequence $\{R_h(\cdot, t_{\varphi_j}), j\in \N \}$. By continuity, $\underline{R}_h(t_{\varphi_j},t_n)$ tends to $(\underline{R}(\cdot,t_n),g)$ as $j$ goes to infinity. Since we have seen that the first term is non-negative for any $j\in \N$, this concludes the proof.
\end{proof}

\noindent Before the proof of Proposition \ref{prop:app1}, we prove a useful technical lemma:
\begin{lemma}\label{lem:app2}
For all $h_1<h_2 \in (0,1/2)$, there exists a constant $\tilde{M}_{h_1}>0$ such that for all $\xi\in W_+^*$,
\begin{equation*}
\int_0^1 \sup_{h\in [h_1,h_2]} \left( K_h^*\xi (u) \right)^2 \ {\rm d}u < \tilde{M}_{h_1} \ \|\xi\|_{H_{h_1}^*}^2 \ .
\end{equation*}
\end{lemma}

\begin{proof}
Recall that $K_h^*\xi \in L^2[0,1]$. In \cite[Chap. 5.1.3]{Nualart}, for $h< 1/2$, $K_h$ is given by the following formula: $\forall s,t \in [0,1]$,
\begin{align}\label{eq:K_h}
 K_h(t,s) &= c_h \left( \left(\frac{t(t-s)}{s}\right)^{-(1/2-h)} + (1/2-h)\ s^{1/2-h} \int_s^t u^{h-3/2} (u-s)^{h-1/2} \ {\rm d}u \right) \mathbf{1}_{[0,t)}(s)\\
&= c_h \left(C_h(t,s) + (1/2-h) D_h(t,s) \right) , \nonumber
\end{align}
where $h\mapsto c_h$ is positive and infinitely differentiable. The second term of this sum is uniformly bounded in $s<t\in [0,1]$ and $h\in [h_1, h_2]$, while the first one diverges when $t$ tends to $s$. Hence for $\epsilon$ small enough, $|t-s|<\epsilon$ implies that $K_h(t,s) \leq K_{h_1}(t,s)$ uniformly in $h\in[h_1, h_2]$, since then,
\begin{align*}
\left(\frac{t(t-s)}{s}\right)^{-(1/2-h)} < \left(\frac{t(t-s)}{s}\right)^{-(1/2-h_1)} \ ,
\end{align*}
and the rest of $K_h(t,s)$ is negligeable compared to this last expression. Hence for $u\in [0,1]$,
\begin{align*}
\sup_{h\in [h_1,h_2]} \left( K_h^*\xi (u) \right)^2 &= \sup_{h\in [h_1,h_2]} \left( \int_u^1 K_h(t,u) \ {\rm d}\xi(t) \right)^2 \\
&= \sup_{h\in [h_1,h_2]} \left( \int_u^{u+\epsilon} K_h(t,u) \ {\rm d}\xi(t) + \int_{u+\epsilon}^1 K_h(t,u) \ {\rm d}\xi(t) \right)^2 \\
&\leq 2 \sup_{h\in [h_1,h_2]} \left( \int_u^{u+\epsilon} K_\eta(t,u) \ {\rm d}\xi(t) \right)^2 + 2 \sup_{h\in [h_1,h_2]} \left(\int_{u+\epsilon}^1 K_h(t,u) \ {\rm d}\xi(t) \right)^2 \ .
\end{align*}
According to the remark that on the set $\{(t,u,h): |t-u|\geq\epsilon , h\in[h_1,h_2]\}$, $K_h(t,u)$ is uniformly bounded, there is a positive constant (possibly depending on $\epsilon$) $M$ such that $K_h(t,u)/K_{h_1}(t,u) \leq M$. Because $\xi$ is a finite nonnegative Radon measure, 
\begin{align*}
\int_{u+\epsilon}^1 K_h(t,u) \ {\rm d}\xi(t) &\leq M \int_{u+\epsilon}^1 K_{h_1}(t,u) \ {\rm d}\xi(t) \ .
\end{align*}
It follows that:
\begin{align*}
\sup_{h\in [h_1,h_2]} \left( K_h^*\xi (u) \right)^2 \leq 2 \left( \int_u^{u+\epsilon} K_{h_1}(t,u) \ {\rm d}\xi(t) \right)^2 + 2 M \left(\int_{u+\epsilon}^1 K_{h_1}(t,u) \ {\rm d}\xi(t)\right)^2 \ ,
\end{align*}
and finally:
\begin{align}
\int_0^1 \sup_{h\in [h_1,h_2]} \left( K_h^*\xi (u) \right)^2 \ {\rm d}u &\leq 2 \|\ K_{h_1}^*\xi \ \|_{L^2}^2 + 2M \|\ K_{h_1}^*\xi\ \|_{L^2}^2 \nonumber\\
&\leq \tilde{M}_{h_1} \| \xi \|_{H_{h_1}^*}^2 \ . \nonumber
\end{align}
A direct consequence of this equation is that the $H_{h_1}^*$-norm is bigger than any other $H_h^*$-norm, when $h\geq h_1$.
\end{proof}

\begin{proof}[Proof of Proposition \ref{prop:app1}]
We first recall that for any $h\in(0,1/2]$, ${\xi\in W^*}$, ${\int_0^1 \left(K_h^*\xi (t)\right)^2 \ {\rm d}t < \infty}$, as well as the facts that $\xi$ is considered as a nonnegative measure, and that $c_h, C_h(\cdot,\cdot)$ and $D_h(\cdot,\cdot)$ are nonnegative quantities. For the sake of readability, we shall use the symbol $'$ to denote the $h$-derivation. For all $s,t\in[0,1]$, 
\begin{equation}\label{eq:app2}
K_h'(t,s) = c_h'K_h(t,s) + c_h \ C_h'(t,s) + c_h\ (1/2-h) D_h'(t,s) - c_h \ D_h(t,s), 
\end{equation}
where
\begin{align*}
&C_h'(t,s) = \log\left(s^{-1}t(t-s)\right) C_h(t,s) \ ,\\
&D_h'(t,s) = (\log s) \ D_h(t,s) + s^{1/2-h}\left(\int_s^t (\log u) u^{h-3/2} \log(u-s) (u-s)^{h-1/2} \ {\rm d}u\right) \ \mathbf{1}_{[0,t]}(s) \ .
\end{align*}

\noindent Each part in the sum of (\ref{eq:app2}) will be treated separately, and each but $C_h'$ using the Cauchy-Schwarz inequality. The first part in (\ref{eq:app2}) gives:
\begin{align*}
\int_0^1 \left( \int_{h_1}^{h_2} c_h' K_h^*\xi(s) \ {\rm d}h \right)^2 \ {\rm d}s \leq (h_2-h_1) \int_0^1 \int_{h_1}^{h_2} \left(c_h'\right)^2 \left(K_h^*\xi(s)\right)^2 \ {\rm d}h \ {\rm d}s \ .
\end{align*}
For $\eta>0$, the infinite differentiability and boundedness away from $0$ of $c_h$ implies that there exists a constant $M^1_{\eta}$ such that for all $h\in [\eta, 1/2-\eta]$, 
\begin{equation*}
\int_0^1 \left(c_h'\int_0^1 K_h(t,s) \ {\rm d}\xi(t) \right)^2 \ {\rm d}s \leq M^1_{\eta} \ \|\xi\|_{H_h^*}^2 \ .
\end{equation*}
Now, Lemma \ref{lem:app2} and Fubini's Theorem imply that:
\begin{equation}\label{eq:borne1}
(h_2-h_1) \int_{h_1}^{h_2} \int_0^1 \left(c_h'\int_0^1 K_h(t,s) \ {\rm d}\xi(t) \right)^2 \ {\rm d}s \ {\rm d}h \leq \tilde{M}^1_{\eta} (h_2-h_1)^2 \|\xi\|_{H_{h_1}^*}^2 \ .
\end{equation}
For the rest of this proof, we might as well consider that $c_h$ is uniformly equal to $1$.

Then for $h\in [\eta, 1/2-\eta]$, we look at the second term in the sum of (\ref{eq:app2}).
Let $\alpha_s\in (s,1]$ such that $s^{-1}\alpha_s(\alpha_s-s) = 1$ if $s\leq 1/2$ and $\alpha_s = 1$ otherwise. Since ${t\mapsto s^{-1}t(t-s)}$ is increasing and maps $[s,1]$ to $[0,s^{-1}(1-s)]$, $\alpha_s$ is uniquely defined. Let some $\nu>0$ such that $h\pm\nu\in (0,1/2)$. Let us remark that $u\in [1,\infty) \mapsto \log(u) u^{-\nu}$ is bounded between $0$ and $(e\ \nu)^{-1}$. Similarly, $u\in (0,1] \mapsto \log(u) u^{\nu}$ is bounded between $-(e\ \nu)^{-1}$ and $0$. Thus for $s\in(0,1)$, the map \[t\in (s,1] \mapsto \log\left(s^{-1}t(t-s)\right) C_{-\nu}(t,s) \mathbf{1}_{\{t> \alpha_s\}} + \log\left(s^{-1}t(t-s)\right) C_{\nu}(t,s) \mathbf{1}_{\{t\leq \alpha_s\}} \] is uniformly bounded (in $s$ \emph{and} $t$) by $-(e\ \nu)^{-1}$ and $(e\ \nu)^{-1}$. We note that when $s\geq 1/2$, the first term in the sum is automatically zero. It follows that:
\begin{align}
\left(\int_{h_1}^{h2} \int_s^1 C_{h}'(t,s) \ {\rm d}\xi(t) \ {\rm d}h \right)^2 &= \Bigg(\int_s^{\alpha_s} \left((s^{-1}t(t-s))^{h_2-h_1}-1\right) C_{h_1}(t,s) \ {\rm d}\xi(t) \nonumber \\
& \quad + \int_{h_1}^{h_2} \int_{\alpha_s}^1 \log\left(s^{-1}t(t-s)\right) C_{-\nu}(t,s) C_{h+\nu}(t,s) \ {\rm d}\xi(t) \ {\rm d}h\Bigg)^2 \nonumber\\
&\leq \left((h_2-h_1) L(h_2-h_1)\right)^2 \times \nonumber\\
&\quad \quad \quad \left(\int_s^{\alpha_s} \frac{(s^{-1}t(t-s))^{h_2-h_1}-1}{(h_2-h_1)L(h_2-h_1)} C_{h_1}(t,s) \ {\rm d}\xi(t) \right)^2 \label{eq:app1Ch1}\\
&\quad + 2(e\ \nu)^{-2}(h_2-h_1) \int_{h_1}^{h_2} \left(\int_{\alpha_s}^{1} C_{h+\nu}(t,s) \ {\rm d}\xi(t) \right)^2 {\rm d}h \ . \label{eq:app1Ch2}
\end{align}
(\ref{eq:app1Ch2}) can be treated easily with Lemma \ref{lem:app2}, since it suffices to choose $\nu = \eta/2$ so that $h+\nu \in [\eta, 1/2-\eta/2]$. Thus, for the same reasons as in (\ref{eq:app2}), we have that:
\begin{align*}
\int_{h_1}^{h_2} \int_0^1 2(e\ \nu)^{-2} \left(\int_{\alpha_s}^1 C_{h+\eta/2}(t,s) \ {\rm d}\xi(t)\right)^2 \ {\rm d}s \ {\rm d}h&\leq 8 (e\ \eta)^{-2} \ \tilde{M}_{h_1}^2 (h_2-h_1) \|\xi\|_{H_{h_1}^*}^2 \nonumber\\
&\leq \tilde{M}_{\eta}^2 (h_2-h_1) \|\xi\|_{H_{h_1}^*}^2\ . 
\end{align*}

(\ref{eq:app1Ch1}) requires more care since the same method would involve $C_{h-\nu}$ with $h-\nu$ occasionally smaller than $h_1$. For all $s\in(0,1)$, we define the application: \[\psi_s(t,h) = \frac{(s^{-1}t(t-s))^{h}-1}{h L(h)}\] on the domain $K_s = \{(t,h): \ s\leq t\leq \alpha_s, \ 0<h\leq 1-2\eta\}$. Because ${s^{-1}t(t-s)\in[0,1]}$, it follows that $\psi_s(t,h) \rightarrow 0$ as $h\rightarrow 0$. Hence $\psi_s$ can be continuously extended to \[K_s^o = \{(t,h): \ s\leq t\leq \alpha_s, \ 0\leq h\leq 1-2\eta\}\ , \] which is compact. It follows from the last two remarks that $\psi_s(t,h)$ is bounded by a constant $\sqrt{M_{\eta}^3}$ which is independent of $s,t$, and $h$. Thus, the term in (\ref{eq:app1Ch1}) is smaller than $M_{\eta}^3 \left(C_{h_1}^*\xi(s)\right)^2$. This finally yields, for the second term of (\ref{eq:app2}):
\begin{align}\label{eq:borne2}
\int_0^1 \left(\int_{h_1}^{h2} \int_s^1 C_{h}'(t,s) \ {\rm d}\xi(t) \ {\rm d}h \right)^2 {\rm d}s &\leq 8(e\ \eta)^{-2}\tilde{M}_{\eta}^2   (h_2-h_1)^2 \|\xi\|_{H_{h_1}^*}^2 \nonumber \\
&\quad \quad + M_{\eta}^3 \left((h_2-h_1) L(h_2-h_1)\right)^2 \|\xi\|_{H_{h_1}^*}^2 \ .
\end{align}

The same technique leads to the following bounds for $D_h'$: first,
\begin{align*}
\int_0^1\left(\int_s^1 (\log s) D_h(t,s) \ {\rm d}\xi(t)\right)^2 \ {\rm d}s &= \int_0^1\left(\int_s^1 (\log s) s^{\nu} \ . s^{-\nu} D_h(t,s)  \ {\rm d}\xi(t)\right)^2 \ {\rm d}s \\
&\leq \int_0^1 (\log s)^2 s^{2\nu} \left(\int_s^1 D_{h+\nu}(t,s) \ {\rm d}\xi(t)\right)^2 \ {\rm d}s \\
&\leq (e\ \nu)^{-2} \int_0^1 \left(\int_s^1 D_{h+\nu}(t,s) \ {\rm d}\xi(t)\right)^2 \ {\rm d}s \ .
\end{align*}
Then, 
\begin{align*}
\int_0^1 \Bigg(s^{1/2-h}&\int_s^1 \int_s^t (\log u) u^{h-3/2} \log(u-s) (u-s)^{h-1/2} \ {\rm d}u \ {\rm d}\xi(t) \Bigg)^2 \ {\rm d}s \\
&\leq \int_0^1 \bigg(s^{1/2-(h-\nu)} s^{\nu} \int_s^1 \int_s^t (\log u) u^{\nu} u^{h-\nu-3/2} \times\\
&\quad \quad \quad\log(u-s) (u-s)^{\nu} (u-s)^{h-\nu-1/2} \ {\rm d}u \ {\rm d}\xi(t)\bigg)^2 \ {\rm d}s \\
&\leq (e\ \nu)^{-2} \int_0^1 \left(s^{1/2-(h-\nu)} \int_s^1 \int_s^t u^{h-\nu-3/2} (u-s)^{h-\nu-1/2} \ {\rm d}u \ {\rm d}\xi(t)\right)^2 \ {\rm d}s \\
&= (e\ \nu)^{-2} \int_0^1 \left(\int_s^1 D_{h-\nu}(t,s) \ {\rm d}\xi(t) \right)^2 \ {\rm d}s \ .
\end{align*}
So that, for $\nu = \eta/2$, the estimates on $D_h$ in the proof of Lemma \ref{lem:app2} imply again:
\begin{equation}\label{eq:borne3}
\int_0^1 \left(\int_s^1 D_h'(t,s) \ {\rm d}\xi(t)\right)^2 \ {\rm d}s \leq 16 (e\ \eta)^{-2} \tilde{M}_{\eta}^4 \|\xi\|_{H_{h_1}^*}^2 \ . 
\end{equation}

All three inequalities (\ref{eq:borne1}),(\ref{eq:borne2}) and (\ref{eq:borne3}), put together with a bound on the last term of (\ref{eq:app2}) (which is easily obtained), end the proof.
\end{proof}

\section{Proof of Lemma \ref{lem:LinInd}}\label{App:lemLinInd}

\begin{proof}
The proof is divided into two cases, depending on whether $h=1/2$ or not. The case $h=1/2$ is immediate since $k_{1/2}(f,g) = \int fg\ dm$. Then \[\forall g\in L^2, \ \lambda_1 k_{1/2}(f_1,g) +\dots + \lambda_n k_{1/2}(f_n,g) = 0 \Rightarrow \lambda_1 f_1 +\dots+ \lambda_n f_n =0\] and this yields $\lambda_1 = \dots =\lambda_n = 0$ because $(f_1,\dots,f_n)$ was assumed to be linearly independent.

\vspace{0.2cm}

In the remaining of this proof, $h\in (0,1/2)$. At first we look at the situation when $n=2$, and the proof is led in two steps, depending on whether $m(f_1^2) = m(f_2^2)$ or not.

Assume first that $m(f_1^2) \neq m(f_2^2)$. Using fractional integration as a linear operator over the indicator functions of the form $\mathbf{1}_{[0,t]}$ straightforwardly implies that for $s\neq t \in [0,1]$, $(R_h(\cdot,t),R_h(\cdot,s))$ is linearly independent. This technique extends to $n\geq 2$ and we shall use it later. Our problem in $L^2(T,m)$ reduces to the aforementionned one via the following trick: let $g\in L^2(T,m)$ be non-zero and orthogonal to $f_1$ and $f_2$. Then, for any $\lambda\in \R$:
\begin{align*}
k_h(f_1,\lambda g) &= \frac{1}{2}\left( m(f_1^2)^{2h} + \lambda^{4h} m(g^2)^{2h} - |m(f_1^2) - \lambda^2 m(g^2)|^{2h} \right) \\
&= R_h(t,u_{\lambda}) ,
\end{align*}
where $t = m(f_1^2)$ and $u_{\lambda} = \lambda^2 m(g^2)$. Let $s=m(f_2^2)$ which is different from $t$ by hypothesis, then the linear independence of $\left(\lambda\mapsto R_h(t,u_\lambda), \lambda\mapsto R_h(s,u_\lambda) \right)$ implies the linear independence of $(k_h(f_1,.),k_h(f_2,.))$ in $H(k_h)$.

Assume now we are in the case of $f_1$ and $f_2$ having the same norm ($\neq 0$) and that $k_h(\cdot,f_1)$ and $k_h(\cdot,f_2)$ satisfy: there is $\lambda \in \R$ such that $k_h(\cdot,f_1) = \lambda k_h(\cdot,f_2)$, ie $\forall g\in L^2(T,m)$,
\begin{equation}\label{eq:khlié}
m(f_1^2)^{2h} - \lambda m(f_2^2)^{2h} = (\lambda-1) m(g^2)^{2h} + m(|f_1-g|^2)^{2h} - \lambda m(|f_2-g|^2)^{2h}.
\end{equation}
Applying this equality to $g=f$, $\lambda$ has to be:
\begin{equation*}
\lambda \ k_h(f_1,f_2)= 2 m(f_1^2)^{2h} \ , 
\end{equation*}
and identically with $g=f_2$, one obtains:
\begin{equation*}
\lambda \ m(f_2^2)^{2h} = \frac{1}{2} k_h(f_1,f_2).
\end{equation*}
Thus $\lambda^2 = 1$. If $\lambda = 1$, this is $m(|f_1-g|^2)^{2h} = m(|f_2-g|^2)^{2h}, \forall g\in L^2$, and we deduce that $f_1=f_2$. Let us prove that $\lambda= -1$ is impossible. Let us consider equation (\ref{eq:khlié}) applied to any $g$ which is orthogonal to $f_1$ and $f_2$ and such that $m(g^2) = m(f_1^2)$:
\begin{align*}
4 m(f_1^2)^{2h} &= m(|f_1-g|^2)^{2h} + m(|f_2-g|^2)^{2h} \\
&= \left(m(f_1^2) + m(g^2)\right)^{2h} + \left(m(f_2^2) + m(g^2)\right)^{2h} \\
&= 2^{2h+1} m(f_1^2)^{2h} ,
\end{align*}
which is impossible whenever $h\neq 1/2$.

\vspace{0.2cm}

In a second step, we extend the result for $n\geq 2$: let $f_1,\dots,f_{n+1} \in L^2$ and assume that $k_h(\cdot,f_{n+1})$ is a linear combination of the family $k_h(\cdot,f_1),\dots,k_h(\cdot,f_n)$. The coefficient in this linear combination are denoted $\left(\lambda_n\right)$. Splitting the maps $f_1,\dots,f_n$ into several groups inside which they have the same norm, we index them differently: $f_{1,1},\dots, f_{1,i_1},\dots,f_{l,1},\dots ,f_{l,i_l}$ where for all $j\in \{1,\dots,l\}$, and all $p,q\in \{1,\dots,i_j \}$, $m(f_{j,p}^2) = m(f_{j,q}^2)$. Then, let $g\in L^2$ be orthogonal to $\textrm{span}\{f_1,\dots,f_{n+1}\}$. We already computed that $k_h(f_i,g) = R_h(m(f_i^2), m(g^2))$. The linear combination is expressed, for all $\mu\in\R$, as follows:
\begin{align*}
k_h(f_{n+1},\mu g) = \sum_{j=1}^l \sum_{k=1}^{i_j} \lambda_{j,k} k_h(f_{j,k},\mu. g) ,
\end{align*}
which is better understood in terms of $R_h$:
\begin{align*}
R_h\left(m(f_{n+1}^2), \mu^2 m(g^2)\right) &= \sum_{j=1}^l \sum_{k=1}^{i_j} \lambda_{j,k} R_h\left(m(f_{j,k}^2), \mu^2 m(g^2) \right) \\
&= \sum_{j=1}^l \left(\sum_{k=1}^{i_j} \lambda_{j,k} \right)\ R_h\left(m(f_{j,1}^2), \mu^2 m(g^2) \right) .
\end{align*}
The linear independence for $R_h$ thus commands that $m(f_{n+1}^2)$ be equal to $m(f_{j,1}^2)$ for some ${j\in\{1,\dots,l\}}$. We will assume, without restriction, that $j=1$. It is then necessary that ${\sum_{k=1}^{i_1} \lambda_{1,k} = 1}$ and that for all $j>1$, $\sum_{k=1}^{i_j} \lambda_{j,k} = 0$. In case $i_1<n$, an induction on $n$ ends the proof. Otherwise, the situation is that $m(f_1^2) = \dots = m(f_{n+1}^2)$ and for all $g\in L^2$:
\begin{equation*}
m\left((f_{n+1}-g)^2\right)^{2h} = \sum_{i=1}^n \lambda_i m\left((f_i-g)^2\right)^{2h} .
\end{equation*}
Because $f_{n+1}$ is linearly independent of $f_1,\dots,f_n$, there exists $g$ orthogonal to every $f_i, i\leq n$ but which is not orthogonal to $f_{n+1}$. Then, the previous equation reads: \[\left(m(f_{n+1}^2)+m(g^2) - 2 m(g\ f_{n+1})\right)^{2h} = \left(m(f_{1}^2)+m(g^2)\right)^{2h}, \] which is impossible due to the fact that $m(g\ f_{n+1})\neq 0$.
\end{proof}

\section{A Fourier-type transform in AWS}\label{App:Fourier}

In this section, it is proved that the operator $\mathcal{F}$ defined in the proof of Lemma \ref{lem:SLND}, satisfies the following, for $x^*, \varphi \in E^*$:
\begin{align*}
\mathcal{F}\varphi(x^*) \neq 0 \Leftrightarrow x^* = \lambda \varphi , \textrm{ for some } \lambda\in \R\setminus \{0\} .
\end{align*}

\noindent First assume that $x^*, \varphi \in E^*$ are linearly independent:
\begin{align*}
I = \int_E \cos\langle x^*,x\rangle \ \langle \varphi,x \rangle \ {\rm d}\mu(x) &= \int_{\R^2} \cos(t_1) \ t_2 \ {\rm d}\mu_{\Sigma}(t_1,t_2) \\
&= \frac{1}{2\pi \sqrt{{\rm det} \Sigma}} \int_{\R^2} \cos(t_1) \ t_2 \ \exp\left(-\frac{1}{2} {\bf t}^{T}\Sigma^{-1} {\bf t}\right) \ {\rm d}\lambda(t_1,t_2)
\end{align*}
where $\Sigma$ represents the covariance structure between the Gaussian random variables $x^*$ and $\varphi$ (defined on the probability space $(E,\mathcal{B}(E),\mu)$). Precisely, 
\begin{equation*}
\Sigma = \left( \begin{array}{cc}
\EE^{\mu}\left(\langle x^*,\cdot \rangle^2\right) & \EE^{\mu}\left(\langle x^*,\cdot \rangle \langle \varphi,\cdot \rangle\right)\\
\EE^{\mu}\left(\langle x^*,\cdot \rangle \langle \varphi,\cdot \rangle\right) & \EE^{\mu}\left(\langle \varphi,\cdot \rangle^2\right)
\end{array} \right)
\end{equation*}
By the linear independence hypothesis on $x^*$ and $\varphi$, $\Sigma$ is not degenerated. Up to renormalization, we can consider that the diagonal in $\Sigma$ is $1$. Let $\gamma$ be the non-diagonal term. Then $I$ reads:
\begin{align*}
I &=  \frac{1}{2\pi \sqrt{{\rm det} \Sigma}} \int_{\R^2} \cos(t_1) \ t_2 \ \exp\left(-\frac{1}{2- 2\gamma^2}(t_1^2 + t_2^2 - \gamma t_1 t_2) \right) \ {\rm d}\lambda(t_1,t_2) \\
 &=  \frac{1}{2\pi \sqrt{{\rm det} \Sigma}} \bigg( I_{t_1> 0, t_2> 0} + I_{t_1> 0,t_2< 0} + I_{t_1< 0,t_2> 0} + I_{t_1<0,t_2<0}\bigg) \\
 &= \frac{1}{2\pi \sqrt{{\rm det} \Sigma}} \bigg( I_{t_1> 0,t_2> 0} + I_{t_1> 0,t_2< 0} - I_{t_1> 0,t_2< 0} - I_{t_1>0,t_2>0}\bigg) \\
 &= 0.
\end{align*}

The converse gives, up to a multiplicative constant, $I = \int_{\R} t \exp\left(it - t^2/2\right) \ {\rm d}t >0$, when $\varphi = x^*$. Thus, whenever $x^*$ and $\varphi$ are linearly dependent, $I$ is non-zero.

\end{appendices}

%\bibliographystyle{amsplain}
%\bibliography{GenProc}

\end{document}